\documentclass[12pt]{amsart}
\usepackage{amsmath,amsfonts,amssymb,color}
\usepackage{a4wide,amsthm}
\usepackage[english]{babel}
\usepackage{textcomp}
\usepackage{hyperref} 
\usepackage{graphicx}
\usepackage{pdfpages}
\renewcommand{\div}{\operatorname{div}}

\usepackage{pgf,tikz}
\usepackage{mathrsfs}
\usetikzlibrary{arrows}
{\newtheorem{thrm}{Theorem}[section]}
{\newtheorem{prpstn}[thrm]{Proposition}}
{\newtheorem{crllr}[thrm]{Corollary}}
{\newtheorem{lmm}[thrm]{Lemma}}
{}
{\newtheorem{rmrk}{Remark}[section]}

\newcommand{\diam}{\operatorname{diam}}

\newcommand{\supp}{\operatorname{supp}}

\begin{document}

\title{A model for suspension of clusters of particle pairs}
\thanks{ This project has received funding from the European Research Council (ERC) under the European Union’s Horizon 2020 research and innovation program Grant agreement No 637653, project BLOC “Mathematical Study of Boundary Layers in Oceanic Motion”. This work was supported by the SingFlows project, grant ANR-18- CE40-0027 of the French National Research Agency (ANR).}
\author{Amina Mecherbet}\address{Sorbonne Universit\'es, Laboratoire Jacques-Louis Lions (UMR 7598), F-75005, Paris, France
}\address{Universit\'e de Paris, Institut de Math\'ematiques de Jussieu-Paris Rive Gauche (UMR 7586), F-75205, Paris, France}

\begin{abstract}
In this paper, we consider $N$ clusters of pairs of particles sedimenting in a viscous fluid. The particles are assumed to be rigid spheres and inertia of both particles and fluid are neglected. The distance between each two particles forming the cluster is comparable to their radii $\frac{1}{N}$ while the minimal distance between the pairs is of order $N^{-1/2}$. We show that, at the mesoscopic level, the dynamics are modelled using a transport-Stokes equation describing the time evolution of the position $x$ and orientation $\xi$ of the clusters. Under the additional assumption that the minimal distance is of order $N^{-1/3}$, we investigate the case where the orientation of the cluster is initially correlated to its position. In this case, a local existence and uniqueness result for the limit model is provided.
\end{abstract}
%
%
\maketitle
\section*{Introduction}
We consider the problem of $N$ rigid particles sedimenting in a viscous fluid under gravitational force. The inertia of both fluid and particles is neglected. At the microscopic level, the fluid velocity and the pressure satisfy a Stokes equation on a perforated domain.
The mathematical derivation of models for suspensions in Stokes flow interested a lot of researches. One of the most investigated question is the effective computations of quantities such as the viscosity or the average sedimentation velocity, see for instance
\cite{Batchelor, Einstein, Feuillebois, GVH, Hasimoto, Haines&Mazzucato, HW, LSP,Niethammer&Schubert, Rubinstein&Keller,SP} and all the references therein.
Regarding the analysis of the associated homogenization problem, it has been proved that the interaction between particles leads to the appearance of a Brinkman force in the fluid equation. This Brinkman force depends on the dilution of the cloud but also the geometry of the particles, see \cite{Allaire, Brinkman,DGR,Feuillebois, Hillairet,HMS, Rubinstein}. In the dynamic case, the justification of a mesoscopic model using a coupled transport-Stokes equation has been proved in \cite{JO} where authors show that the interaction between particles is negligible in the dilute case \textit{i.e.} when the minimal distance between particles is larger than $ \frac{1}{N^{1/3}}$. In \cite{Hofer,Mecherbet} the justification has been extended to regimes that are not so dilute but where the minimal distance between particles is still large compared to the particles radii. The coupled equations derived are:
\begin{equation}\label{VS}
\left\{
\begin{array}{rcl}
\partial_t \rho + \div ( (\kappa g + u)\rho) &=&0\,\\
- \Delta u + \nabla p &=& 6\pi r_0\kappa g \rho\,,\\
\div(u)&=&0.
\end{array}
\right.
\end{equation}
Here $u$ is the fluid velocity, $p$ its associated pressure, $\rho$ is the density of the cloud. $r_0=RN$, where $R$ is the particles radii, $g$ the gravity vector. The velocity $\kappa g = \frac{2}{9}R^2(\rho_p-\rho_f) g$ represents the fall speed of a sedimenting single particle under gravitational force. The derivation of this model is a consequence of the method of reflections which consists in approaching the flow around several particles as the superposition of the flows associated to one particle at time, see \cite{Smo}, \cite[Chapter 8]{KK}, \cite{Luke},  \cite[Section 4]{Guazzelli&Morris}, \cite{LLS}, \cite{HV} for more details.\\

In this paper, we are interested in the case where the cloud is made up of clusters \textit{i.e.} the case where the minimal distance between the particles is proportional to the particles radii $R$. The main motivation is to show the influence of the clusters configuration on the mean velocity fall.
A first investigation in this direction is to consider clusters of pairs of particles where the minimal distance between the particles forming the pair is comparable to their radii. The cluster configuration is determined by the center $x$ and the orientation $\xi$ of the pair. Starting from the microscopic model and assuming the propagation in time of the dilution regime, the first result of this paper is the derivation of a fluid-kinetic model describing the sedimentation of the suspension at a mesoscopic scaling. The fluid-kinetic model obtained couples a Stokes equation for the fluid velocity and pressure $(u,p)$ with a transport equation for the function $\mu(t,x,\xi)$ representing the density of clusters centered in $x$ and having orientation $\xi$ at time $t$, see Theorem \ref{thm}. The mean velocity fall of clusters is formulated through the Stokes resistance matrices while the variation of its orientation involves the gradient of the fluid velocity. In particular, the presence of the gradient of the fluid velocity suggests a similarity with the model of suspension of rod-like particles where the density function depends on the center of the rods $x$ and there orientations $n \in \mathbb{S}^2$, see \cite{DE,HT1, HT2} and the references therein. 

Note that one can reproduce the same arguments for a cloud of clusters constituted of $k \geq 2 $ identical particles by considering the center of mass of the cluster $x=\displaystyle{\frac{1}{k} \underset{p=1}{\overset{k}{\sum}} x_p}$ and $k-1$ orientations  $\xi_q=\frac{x_{q+1} - x_{q}}{2R}$, $q=1,\cdots, k-1$. Extending the assumptions for the general case $k \geq 2$, the limit density $\mu$ and the kinetic equation would depend on $k$ variables. 
However, the relevance of such model can be discussed, in particular with comparison to the models of suspension of polymeric fluids where each polymer is represented by beads connected along a chain. The total number of beads $k$, also called degree of polymerization, can reach for instance $2000$ in the case of ductile materials such as plastic films. Hence, in practice, a fine description of a polymer chain is not suitable and it is more convenient to deal with a coarse model where $k=2$, see \emph{dumbbell} model \cite{DE,LL}.\\

The second result of this paper corresponds to the case where the orientation of the cluster is correlated to its center \textit{i.e.} $\xi=F(t,x)$. Under additional assumptions, see Theorem \ref{thm_bis}, the derived model is a transport equation for $\rho$ coupled to a Stokes equations for  the fluid velocity and pressure $(u,p)$ and a hyperbolic equation for the function  describing the evolution of the cluster orientation $F$. A local existence and uniqueness result for the former system is also presented, see Theorem \ref{thm2}. \\

The starting point is a microscopic model representing sedimentation of $N \in \mathbb{N}^*$ particle pairs in a uniform gravitational field. The pairs are defined as
$$
B^i:= B(x_1^i,R) \cup B(x_2^i,R)\:,\: 1 \leq i \leq N,
$$
where $x_1^i,x_2^i$ are the centers of the $i^{\text{th}}$ pair and  $R$ the radius. We define $(u^N,p^N)$ as the unique solution to the following Stokes problem :
\begin{equation}\label{eq_stokes}
\left \{
\begin{array}{rcl}
-\Delta u^N + \nabla p^N & = &0, \\ 
\div u^N & = & 0,
\end{array}
\text {on $\mathbb{R}^3 \setminus \underset{i=1}{\overset{N}{\bigcup}} \overline{B}^i$, } 
\right.
\end{equation}
completed with the no-slip boundary conditions :
\begin{equation}\label{cab_stokes}
\left \{
\begin{array}{rcl}
u^N & = &  U_1^i \text{ on $ \partial B(x_1^i,R)$}, \\ 
u^N & = &  U_2^i \text{ on $ \partial B(x_2^i,R)$}, \\
\underset{|x| \to \infty}{\lim} |u^N(x)| & = & 0,
\end{array}
\right.
\end{equation}
where $(U_1^i,U_2^i) \in \mathbb{R}^3 \times \mathbb{R}^3  \,,\, 1\leq i \leq N$ are the linear velocities.
In this model, the angular velocity is neglected and we complete the PDE with the motion equation for each couple of particles : 
\begin{equation}\label{ode}
\left\{
\begin{array}{rcl}
\dot{x}_1^i & = & U_1^i, \\ 
\dot{x}_2^i & = & U_2^i.
\end{array}
\right.
\end{equation}
Newton law yields the following equations where inertia is neglected :
\begin{equation}\label{eq_suspension}
\begin{pmatrix}
F_1^i \\ \\
F_2^i
\end{pmatrix} =- \begin{pmatrix}
mg \\ \\
mg
\end{pmatrix},
\end{equation}
where $m$ is the mass of the identical particle adjusted for buoyancy, $g$ the gravitational acceleration, $F_1^i$, $F_2^i$ are the drag forces applied by the fluid on the $i^\text{th}$ particle : 
\begin{eqnarray*}
F_1^i  =  \int_{\partial B(x_1^i,R)} \Sigma(u^N,p^N) n d \sigma &\,,\,&   
F_2^i   = \int_{\partial B(x_2^i,R)} \Sigma(u^N,p^N) n d \sigma,
\end{eqnarray*} 
with $n$ the unit outer normal and $\Sigma(u^N,p^N)=(\nabla u^N + (\nabla u^N)^\top)-p^N \mathbb{I}$ the stress tensor.\\
In order to formulate our results we introduce the main assumptions on the cloud.
\subsection{Assumptions and main results}
We assume that the radius is given by $R=\frac{r_0}{2N}$. In this paper we use the following notations, given a pair of particles $B(x_1,R)$ and $B(x_2,R)$:
\begin{eqnarray*}
x_+:= \frac{1}{2}(x_1+x_2)\:,&\: x_-:= \frac{1}{2}(x_1-x_2)\:,&\: \xi:= \frac{x_-}{R}.
\end{eqnarray*}
Let $T>0$ be fixed. We introduce the empirical density $\mu^N \in \mathcal{P}([0,T]\times \mathbb{R}^3 \times \mathbb{R}^3)$:
 $$
 \mu^N(t,x,\xi) = \frac{1}{N}\sum_1^N \delta_{\left(x_+^i(t),\xi_i(t)\right)}(x,\xi),
 $$
and set $\rho^N$ its first marginal:
\begin{equation}
\rho^N(t,x):= \frac{1}{N} \underset{i}{\sum} \delta_{x_+^i(t)}\,(x).
\end{equation}
We denote by $d_{\min}$ the minimal distance between the centers $x_+^i$:
$$
d_{\min}(t):= {\min}\, \{d_{ij}(t):= |x_+^i(t)- x_+^j(t) |\,,\, i \neq j \}.
$$
We assume that there exists two constants $M_1>M_2>1$ independent of $N$ such that:
\begin{equation}\label{hyp_bound}
M_2 \leq |\xi_i| \leq M_1\: \:,\: i=1,\cdots,N\:\,\: \forall \, t \in [0,T].  
\end{equation}
We assume that $\mu^N$ converges in the sense of measures to $\mu$ \textit{i.e.} for all test function $\psi \in \mathcal{C}_b([0,T] \times \mathbb{R}^3\times \mathbb{R}^3)$ we have:
\begin{equation}\label{conv_mesure}
\int_0^T \int_{\mathbb{R}^3}\int_{\mathbb{R}^3}  \psi(t,x,\xi)\mu^N(t,dx,d\xi) dt \underset{N \to \infty}{\to} \int_0^T \int_{\mathbb{R}^3}\int_{\mathbb{R}^3}  \psi(t,x,\xi)\mu(t,x,\xi)dx\, d\xi\, dt.
\end{equation}
We assume that the first marginal of $\mu$ denoted by $\rho$ is a probability measure such that $\rho \in L^{\infty}(\mathbb{R}^3) \cap L^{1}(\mathbb{R}^3)$.
We introduce $W_\infty(t):=W_\infty(\rho^N(t,\cdot), \rho(t,\cdot))$ the infinite-Wasserstein distance between $\rho^N$ and $\rho$, see \eqref{def_wasserstein} for a definition. We assume that 
\begin{equation}\label{W0}
\underset{t \in [0,T]}{\sup} W_\infty(t) \underset{N \to \infty}{\to} 0.
\end{equation}
For the first result, we assume that there exists a positive constant $\mathcal{E}_1>0$ such that: 
\begin{equation}\label{d_min+}
\underset{t \in [0,T]}{\sup}\,\underset{N \in \mathbb{N}^*}{\sup} \frac{W_\infty^3(t)}{d_{\min}^2(t)} \leq \mathcal{E}_1.
\end{equation}
Regarding the second result, we assume in addition that there exists a positive constant $\mathcal{E}_2>0$ such that:  
\begin{equation}\label{d_min+_bis}
\underset{t \in [0,T]}{\sup}\, \underset{N \in \mathbb{N}^*}{\sup} \frac{W_\infty^3(t)}{d_{\min}^3(t)}\leq \mathcal{E}_2.
\end{equation}
\begin{rmrk}\label{rem1}
Since $\rho \in L^\infty(0;T,L^\infty(\mathbb{R}^3))$, this yields a lower bound for the infinite Wasserstein distance for all $t\in (0,T)$ and all $N\in \mathbb{N}^*$:
\begin{equation}\label{Wasserstein_bound}
\frac{1}{N W_\infty^3(t)} \lesssim \underset{x \in \mathbb{R}^3}{\sup} \frac{\rho^N(t,B(x,W_\infty(t)))}{|B(x,W_\infty(t))|} \lesssim \|\rho\|_{L^\infty(0;T,L^\infty(\mathbb{R}^3))}.
\end{equation}
On the other hand, the definition of the infinite Wasserstein distance ensures that 
\begin{equation}\label{d_min}
W_\infty(t) \geq d_{\min}(t)/2 \,,
\end{equation}
which yields according to \eqref{W0}
\begin{equation}\label{d_min0}
\underset{t \in [0,T]}{\sup} d_{\min}(t) \underset{N \to \infty}{\to} 0.
\end{equation}
Assumption \eqref{d_min+_bis} is only needed for the second Theorem \ref{thm_bis}. Precisely, under assumption \eqref{d_min+}, the minimal distance is at least of order $\frac{C}{\sqrt{N}}$ and $R \ll d_{\min}$. Indeed using \eqref{Wasserstein_bound}, \eqref{d_min+} we have for all $N\in \mathbb{N}^*$
$$
\frac{1}{\sqrt{N}} \leq \|\rho\|_{L^\infty(0;T,L^\infty(\mathbb{R}^3))}^{1/2} \sqrt{\mathcal{E}_2} d_{\min}(t).
$$
Whereas under the additional assumption \eqref{d_min+_bis}, the threshold for the minimal distance is of order $\frac{C}{{N}^{1/3}}$.
\end{rmrk}
Our main results read:
\begin{thrm}\label{thm}
Let $\mu_0 \in L^1(\mathbb{R}^3 \times \mathbb{R}^3)\cap L^\infty (\mathbb{R}^3 \times \mathbb{R}^3)$ a probability measure. Assume that \eqref{hyp_bound}, \eqref{conv_mesure}, \eqref{W0} and \eqref{d_min+} are satisfied. If $r_0\max(\|\rho\|_{L^\infty(0,T; L^\infty(\mathbb{R}^3))},\|\rho\|^{1/3}_{L^\infty(0,T; L^\infty(\mathbb{R}^3))},\|\rho\|^{2/3}_{L^\infty(0,T; L^\infty(\mathbb{R}^3))})$ is small enough, $\mu$ satisfies the following transport equation :
\begin{equation}\label{eq:model}
\left\{
\begin{array}{rcll}
\partial_t \mu + \div_x [ (\mathbb{A}(\xi))^{-1} \kappa g  +  u)\mu]+\div_\xi [ \nabla u \cdot \xi \mu] &=&0\,,& \text{ on } [0,T] \times \mathbb{R}^3 \times \mathbb{R}^3,\\
- \Delta u + \nabla p &=& 6\pi r_0 \kappa \rho g\,, & \text{ on } [0,T] \times\mathbb{R}^3,\\
\div(u)&=&0 \,, & \text{ on } [0,T] \times\mathbb{R}^3,\\
\mu(0,\cdot)&=&\mu_0 \,, & \text{ on } \mathbb{R}^3\times \mathbb{R}^3.
\end{array}
\right.
\end{equation}
\end{thrm}
\begin{rmrk}\label{reg_A_1A_2}
The matrix $\mathbb{A}$ is defined as $\mathbb{A}:=A_1+A_2$ where $A_1$ and $A_2$ are the resistance matrices associated to the sedimentation of a couple of identical spheres, see Section \ref{section_resistance_matrice} for the definition. The term $(\mathbb{A})^{-1} \kappa g $ represents the mean velocity of a couple of identical particles sedimenting under gravitational field. We assume herein that $\mathbb{A}^{-1}, A_1, A_2 \in L^{\infty}(\mathbb{R}^3)$.
\end{rmrk}
\begin{rmrk}\label{rem_mu}
In the case where $\mu_0$ is compactly supported with respect to the second variable $\xi$ uniformly in the first variable $x$, local existence and uniqueness of the above coupled system can be shown following the result of \cite[Chapter 8]{theseHofer} for the model of suspension of rod-like particles. In particular, the $L^1$ norm of the spatial density $\rho$ is conserved in time while the $L^\infty$ norm of $\rho(t,\cdot)$ is bounded by $\underset{x\in \mathbb{R}^3}{\sup}|\supp(\mu(t,x,\cdot) | \|\mu(t,\cdot,\cdot)\|_\infty$. This ensures existence and uniqueness for a small time $T$ and $\|\rho\|_{L^\infty(0,T; L^\infty(\mathbb{R}^3))}$ is controlled by $\|\mu_0\|_\infty$ and $\underset{x\in \mathbb{R}^3}{\sup}|\supp(\mu_0(x,\cdot) |$ .
\end{rmrk}
The second result concerns the case where the vectors along the line of centers $\xi_i$ are correlated to the positions of centers $x_+^i$. 
\begin{thrm}\label{thm_bis}
Assume now that $\rho \in W^{1,\infty}(\mathbb{R}^3) \cap W^{1,1}(\mathbb{R}^3)$, $\mathbb{A}^{-1} \in W^{2,\infty}(\mathbb{R}^3)$ and consider the additional assumption \eqref{d_min+_bis}. Assume that there exists a function $F_0\in W^{1,\infty}(\mathbb{R}^3)$ such that $\xi_i(0)= F_0(x_+^i(0))$ for all $ 1 \leq i \leq N$. There exists $\bar{T}\leq T$ independent of $N$ and unique $F^N \in L^\infty(0,\bar{T}; W^{1,\infty}(\mathbb{R}^3))$ such that for all $t \in[0,\bar{T}]$ we have:
$$
\mu^N= \rho^N \otimes \delta_{F^N}\: \:\text{ and $F^N(0,\cdot)= F_0$}.
$$
Moreover, the sequence $(F^N)_N$ admits a limit $F \in L^\infty(0,\bar{T};W^{1,\infty}(\mathbb{R}^3))$. The limit measure $\mu$ is of the form $\mu= \rho \otimes \delta_{F}$ and the triplet $(\rho, F,u)$ satisfies the following system
\begin{equation}\label{eq_limite}
\left\{
\begin{array}{rcll}
\partial_t F+ \nabla F \cdot (\mathbb{A}(F)^{-1}\kappa g+ u) &=& \nabla u \cdot F,& \text{ on $[0,\bar{T}] \times \mathbb{R}^3,$}\\
\partial_t \rho + \div((\mathbb{A}(F)^{-1}\kappa g+ u)\rho)&=&0,&\text{ on $[0,\bar{T}] \times \mathbb{R}^3,$}\\ 
-\Delta u + \nabla p& =& 6\pi r_0 \kappa g \rho,& \text{ on $\mathbb{R}^3$},\\
\div u &=&0,& \text{ on $\mathbb{R}^3$},\\
\rho(0,\cdot)&=&\rho_0,& \text{ on $\mathbb{R}^3$},\\
F(0,\cdot) & =& F_0& \text{ on $\mathbb{R}^3$}.
\end{array}
\right.
\end{equation}
\end{thrm}
We finish with a local existence and uniqueness result for the limit model.
\begin{thrm}\label{thm2}
Let $p>3$, $F_0 \in W^{2,p}(\mathbb{R}^3)$ and $\rho_0 \in W^{1,p}(\mathbb{R}^3)$ compactly supported. There exists $T>0$ and unique triplet $(\rho, F,u)\in L^\infty(0,T;W^{1,p}(\mathbb{R}^3)) \times L^\infty(0,T;W^{2,p}(\mathbb{R}^3)) \times L^\infty(0,T;W^{3,p}(\mathbb{R}^3))$ satisfying \eqref{eq_limite}.
\end{thrm}
As in \cite{Mecherbet}, the idea of proof of Theorem \ref{thm} and \ref{thm_bis} is to provide a derivation of the kinetic equation satisfied weakly by $\mu^N$. This is done by computing the first order terms of the velocities of each pair:
\begin{equation}\label{velocity_fall}
\left\{
\begin{array}{rcl}
\dot{x}_+^i& \sim &(\mathbb{A}(\xi_i))^{-1} \kappa g + \frac{6\pi r_0}{N} \underset{j\neq i}{\sum} \Phi(x_+^i-x_+^j)\kappa g,\\
\dot{\xi}_i & \sim &  \left( \frac{6\pi r_0}{N} \underset{j\neq i}{\sum} \nabla \Phi(x_+^i-x_+^j)\kappa g \right)\cdot \xi_i  .
\end{array}
\right.
\end{equation}
The interaction force $\Phi$ is the Oseen tensor, see formula \eqref{Oseen}. 
This development is a corollary of the method of reflections which consists in approaching the solution $u^N$ of $2N$ separated particles by the superposition of fields produced by the isolated $2N$ particle solutions. We refer to \cite{Smo}, \cite{Luke}, \cite[Chapter 8]{KK} and \cite[Section 4]{Guazzelli&Morris}, \cite{LLS} for an introduction to the topic. We also refer to \cite{HV} where a converging method of reflections is developed and is used in \cite{Hofer}. In this paper we reproduce the same method of reflections developed in \cite[Section 3]{Mecherbet}. However this method is no longer valid in the case where the minimal distance is comparable to the particle radii. The idea is then to approach the velocity field $u^N$ by the superposition of fields produced by the isolated $N$ couple of particles $B^i=B(x_1^i,R) \sqcup B(x_2^i,R)$. This requires an analysis of the solution of the Stokes equation past a pair of particles. In particular, we need to show that these special solutions have the same decay rate as the Stokeslets, see \cite[Section 2.1]{Mecherbet}. Precisely, in Section 1, we prove that the solution of the Stokes equation past a pair of particles can be approached by the Oseen tensor at first order. \\
The convergence of the method of reflections is ensured under the condition that the minimal distance $d_{\min}$ between the centers $x_+^i$ satisfies 
$$
\frac{W_\infty^3}{d_{\min}}+ \frac{W_\infty^3}{d_{\min}^2} < + \infty\,,
$$
and that the distance $|x_1^i -x_2^i|$ for each pair satisfies formula \eqref{hyp_bound}.\\
In this paper, we focus only on the derivation of the mesoscopic model. Precisely, we do not tackle the propagation in time of the dilution regime and the mean field approximation. We provide in Propositions \ref{conservation_distance} and \ref{boundedness} some estimates showing that the control on the minimal distance $d_{\min}$ depends on the control on the infinite Wasserstein distance $W_\infty(\rho^N,\rho)$.
However, the gradient of the Oseen tensor appearing in equation \eqref{velocity_fall} leads to a log term in the estimates involving the control of $W_\infty(\rho^N,\rho)$, see Proposition \ref{inifnite_cv}. This prevents us from performing a Gronwall argument in order to prove the mean field approximation in the spirit of \cite{Hauray,HJ}.
\subsection{Outline of the paper}
The remaining sections of this paper are organized as follows. In section 1 we present an analysis of the particular solution of two translating spheres in a Stokes flow. The main result of this section is the justification of the approximation of this particular solution using the Oseen tensor and proving some decay properties similar to the Stokeslets.
In section 2 we present and prove the convergence of the method of reflections using the estimate of Appendix A. We also present two particular cases of the application of this method which are useful later.
In section 3 we compute the particle velocities $(\dot{x}_+^i, \dot{\xi}_i)_{1 \leq i \leq N}$ using the estimates provided in the previous section.\\
Section 4 is devoted to the proof of the first Theorem \ref{thm}. Precisely, we prove that the discrete density $\mu^N$ satisfies weakly a transport equation \eqref{ff_VS_N} which can be seen as a discrete version of the limit equation \eqref{eq:model}. In particular, equation \eqref{ff_VS_N} is formulated using a discrete convolution operator $\mathcal{K}^N \rho^N \sim \Phi*\rho^N$ defined rigorously in Section 4. The convergence proof is obtained by showing that $\mathcal{K}^N \rho^N$ converges to the continuous convolution operator $\mathcal{K} \rho = \Phi*\rho$. Convergence estimates of $\mathcal{K}^N\rho^N -\mathcal{K}\rho$ are provided in the Appendix B.\\
Section 5 is devoted to the proof of Theorems \ref{thm_bis} and \ref{thm2}. The first step is to prove local existence and uniqueness results for the correlation function $F^N$ solution of \eqref{correlationN} and also for $F$ the solution of the hyperbolic equation \eqref{correlation}. The idea is to apply a fixed-point argument using some stability estimates provided in the last Appendix C. The last part of Section 5 concerns the convergence of the microscopic model to the mesoscopic model. This convergence result is obtained by showing that the sequence $F^N$ converges is some sense to $F$.
\subsection{Notations}
In this paper, $n$ always refers to the unit outer normal to a surface and $d\sigma $ denotes the measure integration on the surface of the particles.\\
We recall the definition of the Green's function for the Stokes problem $(\mathcal{U},\mathcal{P})$ where $\mathcal{U}$ is also called the Oseen tensor, See \cite[Formula (IV.2.1)]{Galdi} or \cite[Section 2.4.1]{KK}. 
\begin{eqnarray}\label{Oseen}
\Phi(x)= \frac{1}{8\pi} \left (\frac{\mathbb{I}}{|x|}+\frac{x\otimes x}{|x|^3} \right ),& \displaystyle{P(x)= \frac{1}{4\pi} \frac{x}{|x|^3}}.
\end{eqnarray}
Given two probability measures $\nu_1$, $\nu_2$, we define the infinite Wasserstein distance as 
$$
W_\infty(\nu_1,\nu_2)
:={\inf} \Big \{\pi -  \text{esssup} |x-y|\:,\: \pi \in \Pi(\nu_1,\nu_2) \Big  \},
$$
where $\Pi(\nu_1,\nu_2)$ is the set of all probability measures on $\mathbb{R}^3 \times \mathbb{R}^3$ with first marginal $\nu_1$ and second marginal $\nu_2$. In the case where $\nu_1$ is absolutely continuous with respect to the Lebesgue measure, then according to \cite{Champion} the following definition holds true
\begin{equation}\label{def_wasserstein}
W_\infty(\nu_1,\nu_2)
:={\inf} \Big \{\nu_1 -  \text{esssup} |T(x)-x|\:,\: T : \text{supp } \nu_1 \to \mathbb{R}^3 \,,\, \nu_2 =T\# \nu_1\Big  \},
\end{equation} 
In particular, this distance is well adapted to the estimates of the discrete convolution operator $\mathcal{K}^N \rho^N$ defined in \eqref{def_K^N}. Precisely, the infinite Wasserstein distance allows to localise the singularity of the Oseen tensor and is closely related to the minimal distance $d_{\min}$. We refer also to \cite{Hauray,HJ, Champion, Mecherbet} for more details.\\
Given a couple of velocities $(U_1,U_2)\in \mathbb{R}^3 \times \mathbb{R}^3$ we use the following notations
$$
U_+:=\frac{U_1+U_2}{2} , U_-:=\frac{U_1-U_2}{2}.
$$
Finally, in the whole paper we use the symbol $\lesssim$ to express an inequality with a multiplicative constant independent of $N$ and depending only on $r_0$, $ \|\rho\|_{L^\infty(0,T; L^\infty(\mathbb{R}^3))}$, $\mathcal{E}_1$, $\mathcal{E}_2$ and eventually on $\kappa |g|$ which is uniformly bounded, see \cite{Mecherbet}.
\section{Two translating spheres in a Stokes flow}
In this section, we focus on the analysis of the Stokes problem in $\mathbb{R}^3$ past a pair of particles. Given $x_1, x_2 \in \mathbb{R}^3$ and $R_1,R_2>0,$ such that $|x_1-x_2|>R_1+R_2,$ we consider two spheres $B_\alpha:=B(x_\alpha,R_\alpha)$ $\alpha=1,2$ and focus on the following Stokes problem:
\begin{equation}\label{Stokes_pair}
\left \{
\begin{array}{rcl}
-\Delta u + \nabla p & = &0, \\ 
\div u & = & 0,
\end{array}
\text { on $\mathbb{R}^3 \setminus \bar{B}_1\cup \bar{B}_2$, } 
\right.
\end{equation}
completed with the no-slip boundary conditions:
\begin{equation}\label{cab_pair}
\left \{
\begin{array}{rcl}
u & = &  U_\alpha, \text{ on $ \partial B_\alpha$, $\alpha=1,2$}, \\
\underset{|x| \to \infty}{\lim} |u(x)| & = & 0,
\end{array}
\right.
\end{equation}
where $U_\alpha \in \mathbb{R}^3$ for $\alpha=1,2$. Classical results on the Steady Stokes equations for exterior domains (see \cite[Chapter V]{Galdi} for more details) ensure the existence and uniqueness of equations \eqref{Stokes_pair} -- \eqref{cab_pair}.
In this section, we aim to describe the velocity field $u$ in terms of the force applied by the fluid on the particles defined as: 
$$
F_\alpha  :=  \int_{\partial B_\alpha} \Sigma(u,p) n d \sigma \,\:,\: \alpha=1,2.
$$
We refer to the paper \cite{Jeffrey&Onishi} for the following statements. Neglecting angular velocities and torque we emphasize that there exists a linear mapping called resistance matrix satisfying:
\begin{equation}
\begin{pmatrix}
F_1 \\
F_2
\end{pmatrix}
= -3\pi (R_1+R_2)
\begin{pmatrix}
A_{11} & A_{12} \\
A_{21} & A_{22}
\end{pmatrix}
\begin{pmatrix}
U_1 \\
U_2
\end{pmatrix},
\end{equation}
where $A_{\alpha \beta}$, $1\leq \alpha, \beta\leq 2$, are $3\times3$ matrices depending only on the non-dimensionalized centre-to-centre separation $s$ and the ratio of the spheres' radii $\lambda$:
\begin{eqnarray*}
s:= 2 \frac{x_1-x_2}{R_1+R_2},& \displaystyle{\lambda =\frac{R_1}{R_2}},
\end{eqnarray*}
each of these matrices is of the form:
\begin{equation}\label{formule_mat}
A_{\alpha \beta}:= g_{\alpha, \beta}(|s|,\lambda) \mathbb{I} + h_{\alpha, \beta}(|s|,\lambda) \frac{s\otimes s}{|s|^2},
\end{equation}
where $\mathbb{I}$ is the $3\times 3 $ identity matrix and $g_{\alpha, \beta}$, $h_{\alpha, \beta}$ are scalar functions.  We refer to the paper of Jeffrey and Onishi \cite{Jeffrey&Onishi} where the authors provide a development formulas for $g_{\alpha, \beta}$ and $h_{\alpha, \beta}$  given by a convergent power series of $|s|^{-1}$. 
Note that the matrices satisfy 
\begin{equation}\label{symetry}
\begin{array}{rcl}
A_{22}(s,\lambda) &=&A_{11}(s,\lambda^{-1}) , \\
A_{12}(s,\lambda)&=&A_{21}(s,\lambda), \\
A_{12}(s,\lambda)&=&A_{12}(s,\lambda^{-1}).
\end{array}
\end{equation}
Inversly, there exists also a linear mapping called mobility matrix such that 
\begin{equation}
\begin{pmatrix}
U_1 \\
U_2
\end{pmatrix}
=- \frac{1}{3\pi (R_1+R_2)}
\begin{pmatrix}
a_{11} & a_{12} \\
a_{21} & a_{22}
\end{pmatrix}
\begin{pmatrix}
F_1 \\
F_2
\end{pmatrix}.
\end{equation}
The matrices $a_{\alpha,\beta}$ depend on the same parameters as matrices $A_{\alpha,\beta}$ and satisfy a formula analogous to \eqref{formule_mat}. They are also symmetric in the sense of formula \eqref{symetry}. \\
The resistance and mobility matrices satisfy the following formula: 
\begin{equation}\label{inversion}
\begin{pmatrix}
A_{11} & A_{12} \\ A_{21} & A_{22}
\end{pmatrix} \begin{pmatrix}
a_{11} & a_{12} \\ a_{21} & a_{22}
\end{pmatrix} =
\begin{pmatrix}
\mathbb{I} & 0 \\ 0 & \mathbb{I}
\end{pmatrix} ,
\end{equation}
Again, we refer to \cite{Jeffrey&Onishi} for more details.
\subsection{Restriction to the case of two identical spheres}\label{section_resistance_matrice}
We simplify the study by assuming that $R_1=R_2=R$ \textit{i.e.} $\lambda=1$. This means that the resistance matrix depends only on the parameter $s$ which becomes:
$$
s= \frac{x_1-x_2}{R}= 2 \,\xi,
$$ 
and we have: $$A_{22}(s,1)=A_{11}(s,1).$$ 
Hence we reformulate the resistance matrix as follows:
\begin{equation}\label{resistance}
\begin{pmatrix}
F_1 \\
F_2
\end{pmatrix}
= -6\pi R
\begin{pmatrix}
A_{1}(\xi) & A_{2}(\xi) \\
A_{2}(\xi) & A_{1}(\xi)
\end{pmatrix}
\begin{pmatrix}
U_1 \\
U_2
\end{pmatrix},
\end{equation}
and the mobility matrix: 
\begin{equation}
\begin{pmatrix}
U_1 \\
U_2
\end{pmatrix}
= -(6\pi R)^{-1}
\begin{pmatrix}
a_{1}(\xi) & a_{2}(\xi) \\
a_{2}(\xi) & a_{1}(\xi)
\end{pmatrix}
\begin{pmatrix}
F_1 \\
F_2
\end{pmatrix}.
\end{equation}
Formula \eqref{inversion} yields the following relations 
\begin{equation}\label{inversion_bis}
\left\{
\begin{array}{c}
A_1 a_1+A_2 a_2= \mathbb{I}, \\
A_1 a_2+ A_2 a_1 = 0.
\end{array}
\right.
\end{equation}
We are interested in providing a formula for the velocity $u$ and showing some decay properties. In this paper we use the notation $(U[U_1,U_2],P([U_1,U_2])$ for the unique solution to 
\begin{equation*}
\left \{
\begin{array}{rcl}
-\Delta U[U_1,U_2] + \nabla P[U_1,U_2] & = &0, \\ 
\div U[U_1,U_2] & = & 0,
\end{array}
\text { on $\mathbb{R}^3 \setminus \bar{B}_1\cup \bar{B}_2$, } 
\right.
\end{equation*}
completed with the no-slip boundary conditions:
\begin{equation*}
\left \{
\begin{array}{rcl}
U[U_1,U_2] & = &  U_\alpha, \text{ on $ \partial B_\alpha$, $\alpha=1,2$}, \\
\underset{|x| \to \infty}{\lim} |U[U_1,U_2](x)| & = & 0,
\end{array}
\right.
\end{equation*}
Note that there is no ambiguity regarding the dependence of the solution $U[U_1,U_2]$ with respect to $x_+$ and $\xi$. Indeed, in this paper, since we consider the solutions $U[U_1^i,U_2^i]$ associated to each cluster $B_i$ with some velocities $U_1^i,U_2^i$, the dependence with respect to the centers $x_+^i$ and orientations $\xi_i$ is implicitly given by the dependence of the velocities with respect to $i$.
The main result of the section is the following
\begin{prpstn}
Denote by $\xi:= \frac{x_-}{R}$. Assume that there exists $M_1>1$ such that $|\xi| <M_1$. There exists a vector field $\mathcal{R}[U_1,U_2]$ depending on $U_1,U_2,\xi,x_+$ such that for all $|x-x_+|>4M_1R$ we have
\begin{equation}\label{def_formule}
U[U_1,U_2](x)=-\Phi(x_+-x)(F_1+F_2)+\mathcal{R}[U_1,U_2](x),
\end{equation}
Moreover, there exists a positive constant independent of $U_1,U_2,\xi,x_+$ and depending only on $M_1$ such that for all $|x-x_+|>4M_1R$ we have
\begin{equation}\label{decay_rate_R}
\left|\nabla^\beta \mathcal{R}[U_1,U_2](x)\right| \leq C(M_1) R^2 \frac{|U_1|+|U_2|}{|x-x_+|^{2+|\beta|}}, \: \forall \, \beta \in \mathbb{N}^3.
\end{equation}
The unique solution $(U[U_1,U_2],P[U_1,U_2])$ satisfies the following decay property with $C(M_1)$ independent of $x_+$, $\xi$, $U_1$ and $U_2$.
\begin{eqnarray}\label{decay_rate}
\left|\nabla^\beta U[U_1,U_2](x)\right|\leq C(M_1) R \frac{|U_1|+|U_2|}{|x-x_+|^{1+|\beta|}},&\displaystyle{  \left|\nabla^\beta P[U_1,U_2](x)\right|\leq C(M_1) R \frac{|U_1|+|U_2|}{|x-x_+|^{2+|\beta|}},} \: \forall \, \beta \in \mathbb{N}^3.
\end{eqnarray}

\end{prpstn}
\begin{proof}
We consider the case where $x_+=0$ and $R=1$, the generalization to arbitrary $x_+$ and $R$ can be obtained by scaling arguments. In what follows we use the short cut $(u,p):=(U[U_1,U_2],P[U_1,U_2])$ and extend $u$ by $U_\alpha$ on $B(x_\alpha,1)$, $\alpha=1,2$ and we have $u\in\dot{H}^1(\mathbb{R}^3)$. We consider a regular truncation function $\chi=0$ on $B(0,2M_1) \supset B(x_1,1)\cup B(x_2,1)$ and $\chi=1$ on ${}^cB(0,3M_1)$ and we set
\begin{eqnarray*}
\bar{u}&:=&u \chi -\mathfrak{B}_{2M_1,3M_1}[u\nabla \chi],\\
\bar{p}&:=&p\chi,
\end{eqnarray*}
where $\mathfrak{B}_{2M_1,3M_2}$ is the Bogovskii operator on the annulus $B(0,3M_1)\setminus B(0,2M_1)$, see \cite[Appendix A. Lemma 18]{Hillairet} for instance, and satisfies
$$
\div\left(\mathfrak{B}_{2M_1,3M_1}[u\nabla \chi] \right)=u\nabla \chi, $$
$$
 \left \|\mathfrak{B}_{2M_1,3M_1}[u\nabla \chi] \right \|_{L^2(B(0,3M_1)\setminus B(0,2M_1))}\leq C(M_1)\|u\nabla \chi \|_{L^2(B(0,3M_1)\setminus B(0,2M_1))}.
$$
Using Stokes regularity results, see \cite[Theorem IV.4.1]{Galdi}, combined with some Sobolev embeddings we have $\bar{u}\in \mathcal{C}^\infty(\mathbb{R}^3)$ and satisfies a Stokes equation on $\mathbb{R}^3$ with a source term $f=-\div \Sigma(\bar{u},\bar{p})$ having support in  $B(0,3M_1)\setminus B(0,2M_1))$. Hence we can apply the convolution formula with the Green function $\Phi$ and write
$$
\bar{u}(x)=\int_{\mathbb{R}^3}\Phi(x-y)f(y) dy.
$$
Note that $u=\bar{u}$ on ${}^cB(0,3M_1)$. We may then apply a Taylor expansion of $\Phi(\cdot-y)$ for $|x|>3M_1$ and get 
$$
u(x)=\bar{u}(x)=\Phi(x) \int_{\mathbb{R}^3}f(y) dy - \int_{\mathbb{R}^3}\int_0^1(1-t) [\nabla \Phi(x-ty)y]f(y) dy dt.
$$
An integration by parts for the first term yields 
\begin{align*}
\int_{\mathbb{R}^3} f(y) dy &= \int_{B(0,3M_1)\setminus B(0,2M_1))} \div(\Sigma(\bar{u},\bar{p})), \\
&=- \int_{\partial B(0,3M_1)} \Sigma(u,p) n d \sigma,\\
&= -\int_{\partial B(x_1,1)} \Sigma(u,p) n d \sigma+ \int_{\partial B(x_2,1)} \Sigma(u,p) nd \sigma,\\
&=-F_1-F_2,
\end{align*}
we recall that in the above computations the unit normal vector $n$ is pointing outward. It remains to estimate the error term, we recall that using the Bogovskii properties and the embedding $\dot{H}^1(\mathbb{R}^3) \subset L^2_{\text{loc}}(\mathbb{R}^3)$ we have
\begin{multline}
\|f\|_{\dot{H}^{-1}(B(0,3M_1)\setminus B(0,2M_1))}=\|\bar{u}\|_{\dot{H}^1(B(0,3M_1)\setminus B(0,2M_1))}\\ \leq C(M_1) (\| u\|_{\dot{H}^1(\mathbb{R}^3)}+\|u\|_{L^2(B(0,3M_1)\setminus B(0,2M_1))})\leq C(M_1)\| u\|_{\dot{H}^1(\mathbb{R}^3)},
\end{multline}
on the other hand, an integration by parts together with \eqref{resistance} yields
$$
\|\nabla u\|^2_{L^2(\mathbb{R}^3 \setminus (B(x_1,1)\cup B(x_2,2))}=- F_1\cdot U_1-F2\cdot U_2 \leq (\left|A_1(\xi) \right|+\left|A_2(\xi) \right|)^2 (|U_1|+|U_2|)^2.
$$
For the remaining term we introduce $G(x,y)$
$$
G(x,y):=\psi(y)  \int_0^1 (1-t)[\nabla \Phi(x-ty) y] dt,
$$
where $\psi=0$ on ${}^cB(0,7/2 M_1)$ and $\psi=1$ on $B(0,3M_1)$. With this construction and since $\supp f \in B(0,3M_1)\setminus B(0,2M_1)$ we have
$$
\int_{\mathbb{R}^3}\int_0^1(1-t) [\nabla \Phi(x-ty)y]f(y) dy dt = \int_{\mathbb{R}^3} f(y) G(x,y) dy.
$$
Moreover, we have for all $t\in[0,1]$, $|x|>4M_1>7/2 M_1>|y|>2M_1$ 
$$
|x-ty|\geq|x|-t|y|\geq |x|-|y|\geq \frac{1}{8}|x|,
$$
this yields using the decay property of the Oseen tensor for all $|x|>4M_1$ and $y\in B(0,7/2 M_1)\setminus B(0,2M_1)$ 
$$\left\|G(\cdot,x) \right\|_{W^{1,\infty}} \leq \frac{C(M_1)}{|x|^2}. $$
Hence
\begin{align*}
\left|\int_{\mathbb{R}^3} f(y) G(x,y) dy \right|& \leq  \|f\|_{\dot{H}^{-1}(\mathbb{R}^3)}\|G(\cdot,x)\|_{H^1_0(B(0,7/2 M_1)\setminus B(0,2M_1))}\\
&\leq C(M_1) \frac{(\left|A_1(\xi) +A_2(\xi) \right|) (|U_1|+|U_2|)}{|x|^2},
\end{align*}
we conclude by using the fact that $|\xi|\leq M_1$ and the uniform bounds on $A_1, A_2$, see Remark \ref{reg_A_1A_2}.
\end{proof}

 \section{The method of reflections}
In this section, we aim to show that the method of reflections holds true in the special case where the minimal distance and the radius $R$ are of the same order. The idea is to approach the velocity field $u^N$ by the particular solutions developed in the section above. We recall that $u^N$ is the unique solution to the following Stokes problem :
\begin{equation*}
\left \{
\begin{array}{rcl}
-\Delta u^N + \nabla p^N & = &0, \\ 
\div u^N & = & 0,
\end{array}
\text {on $\mathbb{R}^3 \setminus \underset{i=1}{\overset{N}{\bigcup}} \overline{B}^i$, } 
\right.
\end{equation*}
completed with the no-slip boundary conditions :
\begin{equation*}
\left \{
\begin{array}{rcl}
u^N & = &  U_1^i\,, \text{ on $ \partial B(x_1^i,R)$}, \\ 
u^N & = &  U_2^i\,, \text{ on $ \partial B(x_2^i,R)$}, \\
\underset{|x| \to \infty}{\lim} |u^N(x)| & = & 0,
\end{array}
\right.
\end{equation*}
where $(U_1^i,U_2^i) \in \mathbb{R}^3 \times \mathbb{R}^3  \,,\, 1\leq i \leq N$ are such that:
\begin{equation*}
\begin{pmatrix}
F_1^i \\ \\
F_2^i
\end{pmatrix} =- \begin{pmatrix}
mg \\ \\
mg
\end{pmatrix},\:\: \forall \, 1 \leq i \leq N.
\end{equation*}
Thanks to the superposition principle, the sum of the $N$ solutions $\sum_{i=1}^N U[U_1^i, U_2^i]$ satisfies a Stokes equation on $\mathbb{R}^3 \setminus \underset{i=1}{\overset{N}{\bigcup}} \overline{B}^i$, but does not match the boundary conditions. Hence, we define the error term:
$$U[u_*^{(1)}] = u- \sum_{i=1}^N U[U_1^i, U_2^i],$$
which satisfies a Stokes equation on $\mathbb{R}^3 \setminus\underset{i=1}{\overset{N}{\bigcup}} \overline{B}^i$ completed with the following boundary conditions for all $1\leq i \leq N$, $\alpha=1,2$ and $x\in B(x_\alpha^i,R)$ : 
$$
u_*^{(1)}(x) = - \sum_{j\neq i } U[U_1^i, U_2^i](x).
$$
We set then for $\alpha=1,2$ and $1\leq i \leq N$:
$$
U_\alpha^{i,{(1)}}:= u_*^{(1)}(x_\alpha^i),
$$
and reproduce the same approximation to obtain:
$$U[u_*^{(2)} ]:= u- \sum_{i=1}^N \left ( U[U_1^i, U_2^i]+ U[U_1^{i,{(1)}}, U_2^{i,{(1)}}] \right),$$
which satisfies a Stokes equation with the following boundary conditions for all $1\leq i \leq N$, $\alpha=1,2$ and $x\in B(x_\alpha^i,R)$: 
$$
u_*^{(2)}(x) = u_*^{(1)}(x)-u_*^{(1)}(x_\alpha^i)- \sum_{j\neq i } U[U_1^{i,{(1)}}, U_2^{i,{(1)}}](x).
$$
By iterating the process, one can show that for all $k\geq 1$ we have:
$$
u=\sum_{p=0}^{k} \sum_{i=1}^N U[U_1^{i,{(p)}}, U_2^{i,{(p)}}] + U[u_*^{(k+1)}],
$$
where for all $\alpha=1,2$, $1\leq i \leq N$ and $p\geq 0$:
\begin{eqnarray}\label{reflection}
u_*^{(p+1)}(x) & = & u_*^{(p)}(x)-u_*^{(p)}(x_\alpha^i)- \sum_{j\neq i } U[U_1^{i,{(p)}}, U_2^{i,{(p)}}](x)\,, \nonumber \\
u_*^{(0)} & = &  \sum_{i=1}^N U_1^i \, 1_{ B(x_1^i,R)} + U_2^i \,1_{ B(x_2^i,R)}\,, \nonumber\\
U_{\alpha}^{i,{(p)}}& = & u_*^{(p)}(x_\alpha^i)\,, \nonumber \\
U_\alpha^{i,(0)}&=& U_\alpha^i \,.
\end{eqnarray}
The convergence is analogous to the convergence proof in \cite[Section 3.1]{Mecherbet}. We begin by the following estimates that are needed in the computations.
\begin{lmm}\label{maj1}
Under assumptions \eqref{hyp_bound}, \eqref{d_min+} we have for all  $1 \leq  i \neq j \leq N$, $ 1\leq  \beta \leq 2$:
\begin{equation}
|x_+^i -x_\beta ^j|\geq \frac{1}{2}|x_+^i -x_+^j|.
\end{equation}
\end{lmm}
The first step is to show that the sequence  $\underset{i}{\max}(\max( |U_1^{i,(p)}|, |U_2^{i,(p)}|))$ converges when $p$ goes to infinity.
\begin{lmm}\label{conv_serie}
Under assumptions \eqref{hyp_bound}, \eqref{conv_mesure},  \eqref{d_min+} and the assumption that $r_0\|\rho\|_{L^\infty(0,T; L^\infty(\mathbb{R}^3))}^{1/3}$ is small enough, there exists a positive constant $K<1/2$ satisfying for all $1\leq i \leq N$, $p\geq 0$ 
$$
\underset{i}{\max} (\max( |U_1^{i,(p+1)}|, |U_2^{i,(p+1)}|)) \leq K \underset{i}{\max} (\max( |U_1^{i,(p)}|, |U_2^{i,(p)}|)),
$$
for $N$ large enough.
\end{lmm}
\begin{proof}
According to formulas \eqref{decay_rate} and Lemma \ref{maj1}, we have for all $\alpha=1,2$ and $1 \leq i \leq N$:
\begin{align*}
|U_\alpha^{i,(p+1)}| & \leq \left | \sum_{j\neq i }U[U_1^{j,{(p)}}, U_2^{j,{(p)}}](x_\alpha^i) \right |, \\
& \lesssim \frac{C r_0}{N} \left (\sum_{j\neq i }  \frac{1}{d_{ij}}\right ) \underset{j}{\max}\,(|U_1^{j,(p)}|, |U_2^{j,(p)}|),\\
&\leq C r_0\left (\|\rho\|_{L^\infty(0,T; L^\infty(\mathbb{R}^3))} \frac{W_\infty^3}{d_{\min}}+\|\rho\|_{L^\infty(0,T; L^\infty(\mathbb{R}^3))}^{1/3} \right)\,,
\end{align*}
where we used Lemma \ref{JO} for $k=1$. Hence, the first term in the right-hand side vanishes according to \eqref{d_min+} and \eqref{d_min0}.
Finally, if we assume that $r_0\|\rho\|_{L^\infty(0,T; L^\infty(\mathbb{R}^3))}^{1/3} $ is small enough, we obtain the existence of a positive constant $K<1/2$ such that:
$$
\underset{i}{\max} (\max( |U_1^{i,(p+1)}|, |U_2^{i,(p+1)}|)) \leq K 
\underset{i}{\max} (\max( |U_1^{i,(p)}|, |U_2^{i,(p)}|)).
$$
\end{proof}
We have the following result.
\begin{prpstn}\label{convergence}
Under the same assumptions as Lemma \ref{conv_serie}, we have for $N$ large enough:
$$
\underset{k\to \infty}{\lim} \|\nabla  U[u_*^{(k+1)}] \|_2 \lesssim R\, \underset{\alpha=1,2}{\underset{1\leq i \leq N}{\max}} |U_\alpha^{i}|.
$$
\end{prpstn}
\begin{proof}
The proof is analogous to the convergence proof of \cite[Proposition 3.4]{Mecherbet}. This is due to the fact that the particular solutions have the same decay rate as the Oseen-tensor, see \eqref{decay_rate}.
\end{proof}
\subsection{Two particular cases}
\subsubsection{First case}\label{first_case}
Given $W \in \mathbb{R}^3$ we consider in this part $w$ the unique solution to the Stokes equation \eqref{eq_stokes} completed with the following boundary conditions : 
\begin{equation}
w=\left\{ 
\begin{array}{rl}
W& \text {on $B(x_1^1, R),$}\\
-W& \text {on $B(x_2^1, R),$}\\
0 & \text {on $ B(x_1^i, R)\cup B(x_2^i, R)$, $i\neq 1.$}
\end{array}
\right.
\end{equation}
We denote by $\mathcal{W}_\alpha^{i,(p)}$, $\alpha=1,2$, $1\leq i \leq N$, $p \in \mathbb{N}$ the velocities obtained from the method of reflections applied to the velocity field $w$. In other words : 
$$
w = \underset{p=0}{\overset{k}{\sum}} \underset{i}{\sum} U[\mathcal{W}_1^{i,(p)},\mathcal{W}_2^{i,(p)}] + U[w_*^{(k+1)}].
$$
We aim to show that, in this special case, the sequence of velocities $\mathcal{W}_\alpha^{i,(p)}$ and the error term $U[w_*^{(k)}]$ are much smaller than before. This is due to the initial vanishing boundary conditions for $i\neq 1$. Indeed we have :
\begin{prpstn}\label{prop1}
There exists two positive constants $C>0$ and $ L=L(\|\rho\|_{L^\infty(0,T; L^\infty(\mathbb{R}^3))})$ such that for $N$ large enough:
\begin{eqnarray*}
\underset{\alpha=1,2}{\max}\,|\mathcal{W}_\alpha^{i,(p+1)}| &\leq & C (2Cr_0L )^{p} \,  \, \frac{R |x_-^1|}{|x_+^1-x_+^i|^2} \,|W| \:,\: i\neq 1 \:,\: p\geq 0 ,\\
\underset{\alpha}{\max}|\mathcal{W}_\alpha^{1,(p+1)}|&\leq & C 2^{p-1}(r_0CL)^{p} |x_-^1| \frac{R}{d_{\min}}\, |W| \:,\: p\geq 1,\\
\underset{\alpha}{\max}|\mathcal{W}_\alpha^{i,(0)}|+\underset{\alpha}{\max}|\mathcal{W}_\alpha^{1,(1)}|&= & 0\:,\: i\neq 1.
\end{eqnarray*}
\end{prpstn} 
\begin{proof}
We show that the statement holds true for $p=0$ then we prove it for all $p\geq 1$ by induction.
According to formula \eqref{reflection} we have for $p=0$:
$$
\mathcal{W}_\alpha^{1,(0)}= W \delta_{\alpha 1}-W \delta_{\alpha 2},
$$
and for $i\neq 1 $, $\alpha=1,2$, $ U_\alpha^{i,(0)} = 0$. Using \eqref{def_formule}, this yields for $i \neq 1$, $\alpha=1,2$:
\begin{align*}
\mathcal{W}_\alpha^{i,(1)}& = U[\mathcal{W}_1^{1,(0)},\mathcal{W}_2^{1,(0)}](x_\alpha^i),\\
& = -\Phi(x_+^1-x_\alpha^i)(F_1^1+ F_2^1)+ \mathcal{R}[W,-W](x_\alpha^i) ,
\end{align*} 
where: 
\begin{eqnarray*}
F_1^1  =-  6\pi R (A_1(s^1)-A_2(s^1)) W,&
F_2^1 = - 6\pi R (A_2(s^1)-A_1(s^1)) W.
\end{eqnarray*}
Hence, $F_2^1=-F_1^1$ we have then using Lemma \ref{maj1} and the decay rate \eqref{decay_rate_R} for $\mathcal{R}[W,-W]$
\begin{align*}
|\mathcal{W}_\alpha^{i,(1)}|& \lesssim  R^2 \frac{|W|}{d_{i1}^2} \lesssim  R |x_-^1| \frac{|W|}{d_{i1}^2},
\end{align*} 
where we used the fact that the radius $R$ is comparable to $|x_-^1|$ thanks to \eqref{hyp_bound}. Thus, we denote by $C>0$ the maximum between the global constant appearing in \eqref{decay_rate} and the one in the above estimate.\\
This shows that the first statement holds true for $p=0$. For the second estimate we have $|\mathcal{W}_\alpha^{1,(1)}|=0$ and for $p=1$ we have using the decay rate \eqref{decay_rate}
\begin{align*}
|\mathcal{W}_\alpha^{1,(2)}|&= \left | \underset{j\neq 1}{\sum} U[\mathcal{W}_1^{j,(1)},\mathcal{W}_2^{j,(1)}](x_\alpha^1)\right |, \\
&\leq C \underset{j\neq 1}{\sum} \frac{R}{d_{1j}} \max (|\mathcal{W}_1^{j,(1)}|,|\mathcal{W}_2^{j,(1)}|),\\
& \leq C \underset{j\neq 1}{\sum} \left ( \frac{CR^2 |x_-^1|}{d_{1j}^3}   \right ) |W|,\\
&\leq C \frac{|x_-^1|R}{d_{\min}} C r_0 (\mathcal{E}_1\|\rho\|_{L^\infty(0,T; L^\infty(\mathbb{R}^3))}+\|\rho\|_{L^\infty(0,T; L^\infty(\mathbb{R}^3))}^{2/3}  ) \, |W|\,,
\end{align*}
where we used Lemma \ref{JO} for $k=2$ and assumption \eqref{d_min+}. We define then the constant $L>0$ as the constant satisfying:
\begin{multline}\label{defL}
\underset{i}{\max} \, \left (\frac{1}{N} \underset{j\neq i}{\sum} \left (  \frac{1}{d_{ij}^2} \right ) +  \frac{1}{N} \underset{j\neq 1,i}{\sum} \left ( \frac{1}{d_{ij}} + \frac{1}{d_{1j}}\right ) \right )\\ \lesssim \mathcal{E}_1  \|\rho\|_{L^\infty(0,T; L^\infty(\mathbb{R}^3))}+ \|\rho\|_{L^\infty(0,T; L^\infty(\mathbb{R}^3))}^{1/3}+ \|\rho\|_{L^\infty(0,T; L^\infty(\mathbb{R}^3))}^{2/3} := L.
\end{multline}
Now for all $p\geq 1$, $i\neq1$ we have using again \eqref{decay_rate}
\begin{align*}
|\mathcal{W}_\alpha^{i,(p+1)}|&= \left | \underset{j\neq i}{\sum} U[\mathcal{W}_1^{j,(p)},\mathcal{W}_2^{j,(p)}](x_\alpha^i)\right |, \\
& \leq  C \underset{j\neq i}{\sum} \frac{R}{d_{ij}} \max (|\mathcal{W}_1^{j,(p)}|,|\mathcal{W}_2^{j,(p)}|),\\
& \leq C \Big(\underset{j\neq i,1 }{\sum}  \frac{R}{d_{ij}} C (2Cr_0L )^{p-1} \,  \, \frac{R |x_-^1|}{d_{1j}^2}\\
& +\frac{R }{d_{i1}} \frac{R|x_-^1|}{d_{\min}}  C 2^{p-2}(r_0CL)^{p-1} \Big ) |W|,  
\end{align*}
using the fact that $ \displaystyle{\frac{1}{d_{ij} d_{kj}} \leq \frac{1}{d_{ik}} \left (\frac{1}{d_{ij}}+\frac{1}{d_{kj}} \right)}$ we obtain
\begin{align*}
{|\mathcal{W}_\alpha^{i,(p+1)}|}& \leq C \Big ( \frac{R |x_-^1|}{d_{i1}} C (2Cr_0L)^{p-1} \left ( \frac{1}{d_{1i}}\underset{j\neq i,1 }{\sum} \left ( \frac{R}{d_{ij}}+ \frac{R}{d_{1j}} \right ) +  \underset{j\neq i,1 }{\sum} \frac{R}{d_{1j}^2} \right) \\
&+ \frac{R }{d_{i1}} \frac{R|x_-^1|}{d_{\min}}  C 2^{p-2}(r_0CL)^{p-1} \Big)|W|,\\
& \leq C \frac{R |x_-^1|}{d_{i1}} \Big (  C (2Cr_0L\|)^{p-1} \left (\frac{r_0L }{d_{1i}} \right)\\
& + \frac{R }{d_{\min}}  C 2^{p-2}(r_0CL)^{p-1} \Big)|W|,\\
&\leq C  \frac{R |x_-^1|}{d_{i1}^2} \left ( (Cr_0L)^{p} 2^{p-1} + \frac{R d_{i1}}{d_{\min}} C 2^{p-2}(r_0CL)^{p-1}\right) |W|.
\end{align*}
Since $\frac{R d_{1i}}{d_{\min}} \ll r_0L $, the second term can be bounded by $(Cr_0L)^{p}\, 2^{p-2} $ which yields the expected result  because $2^{p-1}+2^{p-2} \leq 2^{p}$.
We prove now the second estimate. Let $p\geq 1$, using the decay rate \eqref{decay_rate} : 
\begin{align*}
|\mathcal{W}_\alpha^{1,(p+1)}|&=\left | \underset{j\neq 1}{\sum} U[\mathcal{W}_1^{j,(p)},\mathcal{W}_2^{j,(p)}](x_\alpha^1)\right | ,\\
& \leq  C \underset{j\neq 1}{\sum} \frac{R}{d_{j1}} \max (|\mathcal{W}_1^{j,(p)}|,|\mathcal{W}_2^{j,(p)}|),\\
&\leq C \left(\underset{j\neq 1 }{\sum}  \frac{R}{d_{1j}} C (2Cr_0L \,  \, \frac{R |x_-^1|}{d_{1j}^2} \right )|W|,\\
& \leq C  (2Cr_0L)^{p-1} C  \frac{R}{d_{\min}} |x_-^1|  \left(\underset{j\neq 1 }{\sum}  \frac{R}{d_{1j}^2} \right ) |W|,\\
& \leq C  2^{p-1} (Cr_0L)^{p}\frac{R}{d_{\min}} |x_-^1| |W|.
\end{align*}
\end{proof}
According to these estimates and the definition of $L$ \eqref{defL}, if we assume that $r_0\max(\|\rho\|_{L^\infty(0,T; L^\infty(\mathbb{R}^3))},\\
\|\rho\|_{L^\infty(0,T; L^\infty(\mathbb{R}^3))}^{1/3},\|\rho\|_{L^\infty(0,T; L^\infty(\mathbb{R}^3))}^{2/3}) $ is small enough to have $2LCr_0<1$ then the following result holds true : 
\begin{crllr}\label{prop2}
Under the assumption that $r_0\max(\|\rho\|_{L^\infty(0,T; L^\infty(\mathbb{R}^3))},\|\rho\|_{L^\infty(0,T; L^\infty(\mathbb{R}^3))}^{1/3},\|\rho\|_{L^\infty(0,T; L^\infty(\mathbb{R}^3))}^{2/3})$ is small enough we have for $N$ large enough
\begin{eqnarray*}
\underset{p=0}{\overset{\infty}{\sum}} \underset{\alpha=1,2}{\max}\,|\mathcal{W}_\alpha^{i,(p)}| &\lesssim &\frac{R |x_-^1|}{|x_+^1-x_+^i|^2} \, |W| \:,\: i\neq 1, \\
\underset{p=1}{\overset{\infty}{\sum}}  \underset{\alpha=1,2}{\max}|\mathcal{W}_\alpha^{1,(p)}|&\lesssim &  \frac{R |x_-^1|}{d_{\min}}  \, |W|.
\end{eqnarray*}
\end{crllr}
This result shows that we can obtain a better estimate for the error term of the method of reflections in this particular case: 
\begin{prpstn}\label{prop3}
We set $\eta:=2CLr_0<1 $ the constant introduced in Proposition \ref{prop1}. For all $i\neq 1 $ we have up to a constant depending on $\|\rho\|_{L^\infty(0,T; L^\infty(\mathbb{R}^3))}$
\begin{eqnarray*}
\|\nabla w_*^{(k)}\|_{L^\infty(B_i)} &\lesssim&   \frac{R|x_-^1|}{d_{i1}^3}  |W|, \\
\| w_*^{(k+1)}\|_{L^\infty(B_i)}  &\lesssim&  R  \|\nabla w_*^{(k)}\|_{L^\infty(B_i)} +  \frac{R}{d_{1i}^2}|x_-^1| \eta^{k-1} |W|.
\end{eqnarray*}
And for $i=1$ we have :
\begin{eqnarray*}
\|\nabla w_*^{(k)}\|_{L^\infty(B_1)}& \lesssim&  \frac{R}{d_{\min}} |x_-^1| \left(\frac{W_\infty^3}{d_{\min}^3}+ |\log W_\infty | \right) |W|, \\
\| w_*^{(k+1)}\|_{L^\infty(B_1)}& \lesssim&R \|\nabla w_*^{(k)}\|_{L^\infty(B_1)} +  \frac{R}{d_{\min}} |x_-^1| \eta^{k-1}|W|,
\end{eqnarray*}

\end{prpstn}
\begin{proof}
\textbf{Estimate for $\|\nabla w_*^{(k)}\|_\infty$.}\\
Let $x \in  B(x_\alpha^i,R)$, with $\alpha=1,2$ and $i\neq 1 $, formula \eqref{reflection} yields: 
\begin{align*}
|\nabla w_*^{(k+1)}(x)| & \leq |\nabla w_*^{(k)}(x)| +  \underset{j\neq i }{{\sum}}  |\nabla  U[\mathcal{W}_1^{j,(k)},\mathcal{W}_2^{j,(k)}](x)| , \\
& \leq \underset{p=0}{\overset{k}{\sum}}\underset{j\neq i }{{\sum}}  |\nabla  U[\mathcal{W}_1^{j,(p)},\mathcal{W}_2^{j,(p)}](x)|,  \\
& \leq \underset{p=0}{\overset{k}{\sum}}\underset{j\neq i,1 }{{\sum}}  |\nabla  U[\mathcal{W}_1^{j,(p)},\mathcal{W}_2^{j,(p)}](x)| +\underset{p=1}{\overset{k}{\sum}}  |\nabla  U[\mathcal{W}_1^{1,(p)},\mathcal{W}_2^{1,(p)}](x)|\\
& +  |\nabla  U[\mathcal{W}_1^{1,(0)},\mathcal{W}_2^{1,(0)}](x)|.  \\
\end{align*}
We estimate the first term applying Corollary \ref{prop2} and the same arguments as before
\begin{align*}
\underset{p=0}{\overset{k}{\sum}}\underset{j\neq i,1 }{{\sum}}  |\nabla  U[\mathcal{W}_1^{j,(p)},\mathcal{W}_2^{j,(p)}](x)|& 
\leq C\underset{p=0}{\overset{k}{\sum}}\underset{j\neq i,1 }{{\sum}} \left ( \frac{R}{d_{ij}^2}\right) \underset{\alpha=1,2}{\max}\, |W_\alpha^{j,(p)}|,\\
&\lesssim \underset{j\neq i,1 }{{\sum}} \left ( \frac{R}{d_{ij}^2} \frac{R|x_-^1|}{d_{1j}^2} \right)|W|, \\
&\lesssim\frac{R |x_-^1|}{d_{1i}^2} \underset{j\neq i,1 }{{\sum}} \left ( \frac{R}{d_{ij}^2}+  \frac{R}{d_{1j}^2} \right)|W|, \\
&\lesssim \frac{R |x_-^1|}{d_{1i}^2} |W|. \\
\end{align*}
We reproduce the same for the second term applying Corollary \ref{prop2}: 
\begin{align*}
\underset{p=1}{\overset{k}{\sum}}  |\nabla  U[\mathcal{W}_1^{1,(p)},\mathcal{W}_2^{1,(p)}](x)| &\leq 
C\underset{p=1}{\overset{k}{\sum}} \left ( \frac{R}{|x_+^1-x_+^i|^2}\right) \max( |\mathcal{W}_1^{1,(p)}|,|\mathcal{W}_2^{1,(p)}|),\\
&\lesssim \frac{R}{|x_+^1-x_+^i|^2} \frac{R}{d_{\min}} |x_-^1|  |W|.
\end{align*}
For the last term, according to \eqref{def_formule} we have : 
\begin{align*}
 \nabla  U[\mathcal{W}_1^{1,(0)},\mathcal{W}_2^{1,(0)}](x) & = - \nabla  \Phi(x_+^1-x)(F_1^1 +F_2^1) +\nabla  \mathcal{R}[\mathcal{W}_1^{1,(0)},\mathcal{W}_2^{1,(0)}](x),
\end{align*}
as $(\mathcal{W}_1^{1,(0)},\mathcal{W}_2^{1,(0)})=(W,-W)$ we have:
\begin{equation*}
\left\{
\begin{array}{rcl}
F_1^1 &=& -6\pi R(A_1(\xi_1)W-A_2(\xi_1)W),\\
F_2^1 &=&- 6\pi R(A_2(\xi_1)W-A_1(\xi_1)W).
\end{array}
\right.
\end{equation*}
Thus $F_2^1=-F_1^1$ and we obtain using the decay rate of $R$ \eqref{decay_rate_R} together with the fact that $|x_-^1|$ is comparable to $R$ thanks to assumption \eqref{hyp_bound} :
\begin{equation*}
 \left | \nabla  U[\mathcal{W}_1^{1,(0)},\mathcal{W}_2^{1,(0)}](x) \right| \lesssim \frac{R |x_-^1|}{|x_+^1-x_+^i|^3}|W| ,\:\forall\, x\in B(x_\alpha^i,R), i\neq1
\end{equation*}
Gathering all the inequalities we have for $i\neq 1$:
\begin{align*}
\|\nabla w_*^{(k)}\|_{L^\infty(B_i)} &\lesssim \frac{R|x_-^1|}{|x_+^1-x_+^i|^3} |W|.
\end{align*}
Analogously for $i=1$ we apply Lemma \ref{JO} for $k=3$ and obtain up to a constant depending on $\|\rho\|_{L^\infty(0,T; L^\infty(\mathbb{R}^3))}$: 
\begin{align*}
|\nabla w_*^{(k+1)}(x)| & \leq |\nabla w_*^{(k)}(x)| +  \underset{j\neq 1 }{{\sum}}  |\nabla  U[\mathcal{W}_1^{j,(k)},\mathcal{W}_2^{j,(k)}](x)| , \\
& \leq \underset{p=0}{\overset{k}{\sum}}\underset{j\neq 1 }{{\sum}}  |\nabla  U[\mathcal{W}_1^{j,(p)},\mathcal{W}_2^{j,(p)}](x)|,  \\
& \leq C\underset{p=0}{\overset{k}{\sum}}\underset{j\neq 1 }{{\sum}} \left ( \frac{R}{d_{1j}^2}\right) \max( |\mathcal{W}_1^{j,(p)}|,|\mathcal{W}_2^{j,(p)}|),\\
& \lesssim \underset{j\neq 1 }{{\sum}} \left ( \frac{R}{d_{1j}^2} \frac{R|x_-^1|}{d_{1j}^2} \right)|W|,\\
&\lesssim \frac{R|x_-^1|}{d_{\min}} \left( \frac{W_\infty^3}{d_{\min}^3} +| \log W_\infty| \right) |W|.
\end{align*}
\textbf{Estimate for $\| w_*^{(k)}\|_\infty$.}
Let $x\in B(x_\alpha^i,R)$, $\alpha=1,2$, $i\neq 1$. We have according to formula \eqref{reflection} : 
\begin{align*}
| w_*^{(k+1)}(x)| & = \left | w_*^{(k)}(x)- w_*^{(k)}(x_\alpha^i)- \underset{j\neq i}{\sum} U[\mathcal{W}_1^{j,(k)},\mathcal{W}_2^{j,(k)}](x)  \right|,\\
& \leq R \|\nabla w_*^{(k)}\|_\infty + \underset{j\neq i}{\sum} \left| U[\mathcal{W}_1^{j,(k)},\mathcal{W}_2^{j,(k)}](x) \right|,\\
&\leq R \|\nabla w_*^{(k)}\|_\infty + C \underset{j\neq i}{\sum} \frac{R}{d_{ij}} \max ( |\mathcal{W}_1^{j,(k)}|,|\mathcal{W}_2^{j,(k)}|) ,  \\ 
& \lesssim R \|\nabla w_*^{(k)}\|_\infty +  \left ( \underset{j\neq i,1}{\sum} \frac{R}{d_{ij}} \eta^{k-1} \frac{R}{d_{1j}^2} + \frac{R}{d_{1i}} \eta^{k-1} \frac{R}{d_{\min}}  \right )|x_-^1| |W|.
\end{align*}
where $\eta = 2C r_0 L <1$ is the constant appearing in Proposition \ref{prop1}. 
Reproducing the same computations as before yields: 
$$
\| w_*^{(k+1)}\|_{L^\infty(B_i)}  \lesssim R  \|\nabla w_*^{(k)}\|_\infty + \frac{R}{d_{1i}^2}|x_-^1| \eta^{k-1} |W|.
$$
In the case $i=1$ we have: 
\begin{align*}
| w_*^{(k+1)}(x)| & = \left | w_*^{(k)}(x)- w_*^{(k)}(x_\alpha^i)- \underset{j\neq
i}{\sum} U[\mathcal{W}_1^{j,(k)},\mathcal{W}_2^{j,(k)}](x)  \right|,\\
&\leq R \|\nabla w_*^{(k)}\|_\infty + C \underset{j\neq 1}{\sum} \frac{R}{d_{1j}} \max ( |\mathcal{W}_1^{j,(k)}|,|\mathcal{W}_2^{j,(k)}|) ,  \\ 
& \lesssim R \|\nabla w_*^{(k)}\|_\infty +  \underset{j\neq 1}{\sum} \frac{R}{d_{1j}} \eta^{k-1} \frac{R}{d_{1j}^2}|x_-^1| |W|, \\
& \lesssim R \|\nabla w_*^{(k)}\|_\infty + \frac{R}{d_{\min}} |x_-^1| \eta^{k-1}|W|.
\end{align*}
\end{proof}
Thanks to these estimates we have the following convergence rate: 
\begin{prpstn}\label{prop4}
$$
\underset{k\to \infty}{\lim} \| \nabla U[w_*^{(k+1)}] \|_{2} \lesssim {R} |x_-^1||W|.$$
\end{prpstn}
\begin{proof}
Reproducing exactly the same proof as in \cite[Proposition 3.4]{Mecherbet}, the main difference appears in the last estimate where we apply Proposition \ref{prop3}: 
\begin{align*}
 \| \nabla U[w_*^{(k+1)}] \|_{2}^2 &\lesssim R^3 \underset{i}{\sum} \left ( \| \nabla w_*^{(k+1)}\|_{L^\infty(B_i)} + \frac{1}{R} \| w_*^{(k+1)}\|_{L^\infty(B_i)} \right)^2,\\
 &\lesssim  R^3 \Big [ \underset{i\neq1}{\sum} \left(    \frac{R^2}{d_{1i}^6} + \frac{1}{d_{1i}^4} \eta^{2(k-1)} \right)\\
 & +  \frac{R^2}{d_{\min}^2} \left(\frac{W_\infty^3}{d_{\min}^3}+| \log W_\infty| \right)^2 + \frac{1}{d_{\min}^2} \eta^{2(k-1)} \Big] |x_-^1|^2|W|^2,\\
 &\lesssim   \left( \frac{R^4}{d_{\min}^3} +\frac{R^2}{d_{\min}} \eta^{2(k-1)} \right)\left (\frac{W_\infty^3}{d_{\min}^3}+| \log W_\infty| \right) |x_-^1|^2|W|^2\\
 & + |x_-^1|^2|W|^2  \frac{R^5}{d_{\min}^2}\left (\frac{W_\infty^3}{d_{\min}^3}+| \log W_\infty| \right)^2+\frac{R^3}{d_{\min}^2} \eta^{2(k-1)} |x_-^1|^2|W|^2\,.
\end{align*}
Taking the limit when $k$ goes to infinity we get:
\begin{align*}
 \| \nabla U[w_*^{(k+1)}] \|_{2}^2 &\lesssim R^2 |x_-^1|^2|W|^2 \left\{\frac{R^2}{d_{\min}^3} \left (\frac{W_\infty^3}{d_{\min}^3}+| \log W_\infty| \right) +   \frac{R^3}{d_{\min}^2}\left (\frac{W_\infty^3}{d_{\min}^3}+| \log W_\infty| \right)^2 \right\}.
\end{align*}
The term inside brackets is bounded as follows:
\begin{multline*}
\frac{R^2}{d_{\min}^3} \left (\frac{W_\infty^3}{d_{\min}^3}+| \log W_\infty| \right)+  \frac{R^3}{d_{\min}^2}\left (\frac{W_\infty^3}{d_{\min}^3}+| \log W_\infty| \right)^2 \\ \leq \frac{R^2}{d_{\min}^2} \frac{W_\infty^3}{d_{\min}^2}+ R | \log W_\infty|+ \frac{R}{d_{\min}^2} \left(  \frac{R}{d_{\min}} \frac{W_\infty^3}{d_{\min}^2} + R | \log W_\infty| \right)^2\,,
\end{multline*}
we recall that $\frac{R}{d_{\min}} < +\infty$ and $\frac{R}{d_{\min}^2} \leq \frac{r_0}{2} \|\rho\|_\infty \frac{W_\infty^3}{d_{\min}^2}$ according to \eqref{Wasserstein_bound}.
\end{proof}

\subsubsection{Second case}\label{second_case}
Given $W \in \mathbb{R}^3$ we consider in this part $w$ the unique solution to the Stokes equation \eqref{eq_stokes} completed with the following boundary conditions : 
\begin{equation}
w=\left\{ 
\begin{array}{rl}
W& \text {on $B(x_1^1, R),$}\\
W& \text {on $B(x_2^1, R),$}\\
0 & \text {on $ B(x_1^i, R)\cup B(x_2^i, R)$, $i\neq 1.$}
\end{array}
\right.
\end{equation}
Denote by $\mathcal{W}_\alpha^{i,(p)}$, $\alpha=1,2$, $1\leq i \leq N$, $p \in \mathbb{N}$ the velocities obtained from the method of reflections applied to the velocity field $w$. In other words : 
$$
w = \underset{p=0}{\overset{\infty}{\sum}} \underset{i}{\sum} U[\mathcal{W}_1^{i,(p)},\mathcal{W}_2^{i,(p)}] + O(R).
$$
We aim to show that, in this special case, the sequence of velocities $\mathcal{W}_\alpha^{i,(p)}$ are also smaller than the general case. This is due to the initial boundary conditions which vanish for $i\neq 1$. Indeed we have :

\begin{prpstn}\label{prop1_bis}
There exists two positive constants $C>0$ and $L=L(\|\rho\|_{L^\infty(0,T; L^\infty(\mathbb{R}^3))})$ such that  :
\begin{eqnarray*}
\underset{\alpha=1,2}{\max}\,|\mathcal{W}_\alpha^{i,(p+1)}| &\leq & C (2Cr_0L)^{p} \,  \, \frac{R}{|x_+^1-x_+^i|} \,|W| \:,\: i\neq 1 \:,\: p\geq 0, \\
\underset{\alpha}{\max}|\mathcal{W}_\alpha^{1,(p+1)}|&\leq & C 2^{p-1}(r_0CL)^{p} R\,|W| \:,\: p\geq 1,\\
\underset{\alpha}{\max}|\mathcal{W}_\alpha^{1,(1)}|&= & 0,
\end{eqnarray*}
for $N$ large enough.
\end{prpstn} 
\begin{proof}
The proof is analogous to the one of Proposition \ref{prop1}.
\end{proof}
According to these estimates, if we assume that $r_0\max(\|\rho\|_{L^\infty(0,T; L^\infty(\mathbb{R}^3))},\|\rho\|_{L^\infty(0,T; L^\infty(\mathbb{R}^3))}^{1/3},\\ \|\rho\|_{L^\infty(0,T; L^\infty(\mathbb{R}^3))}^{2/3})$ is small enough to have $2LCr_0<1$ then the following result holds true: 
\begin{crllr}\label{prop2_bis}
We have for $N$ large enough:
\begin{eqnarray*}
\underset{k=0}{\overset{\infty}{\sum}} \underset{\alpha=1,2}{\max}\,|\mathcal{W}_\alpha^{i,(p+1)}| &\lesssim &\frac{R}{|x_+^1-x_+^i|} \, |W| \:,\: i\neq 1, \\
\underset{k=0}{\overset{\infty}{\sum}}  \underset{\alpha}{\max}|\mathcal{W}_\alpha^{1,(p+1)}|&\lesssim &  R\, |W|.
\end{eqnarray*}
\end{crllr}
\section{Extraction of the first order terms for the velocities}
This section is devoted to the computation of the velocities $U_+^i, U_-^i$ for $1 \leq i \leq N$.
The idea of proof is to apply the method of reflections to the velocity field $u^N$ as presented above and we set :
$$
\underset{p=0}{\overset{\infty}{\sum}} U_\alpha^{i,(p)} = U_\alpha^{i,\infty},1\leq \alpha \leq 2\,,\, 1 \leq i \leq N,
$$
we also use the following notations for the forces associated to the solutions $U[ U_1^{i,\infty}, U_2^{i,\infty}]$:
\begin{eqnarray}\label{forces_formula}
F_1^{i,\infty}& =&- 6\pi R (A_1(\xi_i) U_1^{i,\infty}+ A_2 (\xi_i)U_2^{i, \infty}), \nonumber \\
F_2^{i,\infty}& =&- 6\pi R (A_2(\xi_i) U_1^{i,\infty}+ A_1 (\xi_i)U_2^{i, \infty}). 
\end{eqnarray}
\subsection{Preliminary estimates}
\begin{prpstn}\label{velocity_mr_bis}
If  $r_0\|\rho\|_{L^\infty(0,T; L^\infty(\mathbb{R}^3))},\|\rho\|_{L^\infty(0,T; L^\infty(\mathbb{R}^3))}^{1/3},\|\rho\|_{L^\infty(0,T; L^\infty(\mathbb{R}^3))}^{2/3})$ is small enough and assumptions \eqref{hyp_bound}, \eqref{conv_mesure}, \eqref{d_min+} hold true we have for $N$ large enough and for all $1\leq i \leq N$
\begin{multline*}
\frac{U_1^i+U_2^i}{2}= (A_1(\xi_i)+A_2(\xi_i))^{-1} \frac{m}{6\pi R} g \\+\frac{1}{2} \underset{j \neq i}{\sum}\left ( U[U_1^{j,\infty},U_2^{j,\infty}](x_1^i)+U[U_1^{j,\infty},U_2^{j,\infty}](x_2^i) \right )  + O(\sqrt{R})\underset{\alpha=1,2}{\underset{1\leq i \leq N}{\max}} |U_\alpha^{i}|.
\end{multline*}
$$
\frac{U_1^{i,\infty}+U_2^{i,\infty}}{2}=(A_1(\xi_i)+A_2(\xi_i))^{-1} \frac{m}{6\pi R} g  +  O(\sqrt{R})\underset{\alpha=1,2}{\underset{1\leq i \leq N}{\max}} |U_\alpha^{i}|.
$$
\end{prpstn}
\begin{proof}
We prove the formula for $i=1$ and the same holds true for all $1\leq i \leq N$.
We set $w$ the unique solution to the Stokes equation \eqref{eq_stokes} completed with the following boundary conditions : 
\begin{equation}
w=\left\{ 
\begin{array}{rl}
W& \text {on $B(x_1^1, R),$}\\
W& \text {on $B(x_2^1, R),$}\\
0 & \text {on $ B(x_1^i, R)\cup B(x_2^i, R)$, $i\neq 1,$}
\end{array}
\right.
\end{equation}
with $W$ an arbitrary vector of $\mathbb{R}^3$. We use the method of reflections to obtain :
\begin{align*}
2 mg \cdot W & = 2\int D (u^N) : \nabla w \\
& = -(F_1^{1,\infty} + F_2^{1,\infty}) \cdot W + \underset{k \to \infty}{\lim}2 \int D \left( U[u_*^{(k+1)}] \right): \nabla w .
\end{align*}
For the last term we apply again the method of reflections to the velocity field $w$, see Section \ref{second_case}. We set: $$w_1=\underset{p=0}{\overset{k}{\sum}} \underset{i=1}{\overset{N}{\sum}} U[\mathcal{W}_1^{i,(p)},\mathcal{W}_2^{i,(p)}]:=\underset{i=1}{\overset{N}{\sum}} U[\mathcal{W}_1^{i,\infty},\mathcal{W}_2^{i,\infty}], $$
with 
$$
\| \nabla w- \nabla w_1\|_{L^2( \mathbb{R}^3 \setminus \overline{\underset{i}{\bigcup} B_i})} \leq R |W|.
$$ 
We obtain :
$$
2\int D \left( U[u_*^{(k+1)}]\right): \nabla w  = 2\int \nabla U[u_*^{(k+1)}]: D \left( w_1\right) +2 \int D \left( U[u_*^{(k+1)}]\right): \nabla (w-w_1).
$$
Thanks to the method of reflections, the second term on the right hand side can be bounded by $R^2|W| \underset{\alpha=1,2}{\underset{1 \leq i \leq N}{\max}} |U_\alpha^i|$ (see Proposition \ref{convergence}). 
For the first term, direct computations using \eqref{resistance} show that
\begin{align*}
\left \|\nabla w_1 \right\|_{L^2(\mathbb{R}^3 \setminus \overline{\underset{i}{\bigcup} B_i})}& \leq \underset{i}{\sum} \left \|\nabla  U[\mathcal{W}_1^{i,\infty},\mathcal{W}_2^{i,\infty}] \right\|_{L^2(\mathbb{R}^3 \setminus B_i)},\\
&= \underset{i}{\sum} \big( - \int_{\partial B(x_1^i,R)\cup\partial B(x_2^i,R) }\Sigma\left(U[\mathcal{W}_1^{i,\infty},\mathcal{W}_2^{i,\infty}],P[\mathcal{W}_1^{i,\infty},\mathcal{W}_2^{i,\infty}] \right) n \cdot U[\mathcal{W}_1^{i,\infty},\mathcal{W}_2^{i,\infty}]d \sigma \big)^{1/2}\\
& =\underset{i}{\sum} \Big( 6\pi R \left(A_1(\xi_i) \mathcal{W}_1^{i,\infty}+A_2(\xi_i) \mathcal{W}_2^{i,\infty}  \right)\cdot \mathcal{W}_1^{i,\infty}\\
&+6\pi R \left(A_2(\xi_i) \mathcal{W}_1^{i,\infty}+A_1(\xi_i) \mathcal{W}_2^{i,\infty}  \right)\cdot \mathcal{W}_2^{i,\infty} \Big)^{1/2}\\
&\leq C \sqrt{R}\underset{i}{\sum} \left(\left|\mathcal{W}_1^{i,\infty} \right|+\left|\mathcal{W}_2^{i,\infty} \right| \right).
\end{align*}
Using Corollary \ref{prop2_bis} we get that $ \left \|\nabla w_1 \right\|_{L^2(\mathbb{R}^3 \setminus \overline{\underset{i}{\bigcup} B_i})} \leq C \sqrt{R} |W|$. Finally, we have: 
\begin{align*}
2 mg \cdot W = -(F_1^{1,\infty}+F_2^{1,\infty}) \cdot W + O(R \sqrt{R})|W|\underset{\alpha=1,2}{\underset{1 \leq i \leq N}{\max}} |U_\alpha^i|.
\end{align*}
This being true for all $W\in \mathbb{R}^3$ it yields:
\begin{align*}
2 mg  = -(F_1^{1,\infty}+F_2^{1,\infty}) + O(R\sqrt{R})\underset{\alpha=1,2}{\underset{1 \leq i \leq N}{\max}} |U_\alpha^i|.
\end{align*}
Using the definitions of $F_1^{1,\infty}$ and $F_2^{1,\infty}$, see \eqref{forces_formula}, this becomes:
$$
2 mg = 6\pi R (A_1(\xi_1)+A_2(\xi_1))(U_1^{1,\infty}+U_2^{1,\infty}) + O(R\sqrt{R})\underset{\alpha=1,2}{\underset{1 \leq i \leq N}{\max}} |U_\alpha^i|.
$$
Recall that $A_1(\xi)$ and $A_2(\xi)$ are of the form $h_1(|\xi|) \mathbb{I} + h_2(|\xi|) \frac{\xi \otimes \xi}{|\xi|^2} $. Moreover, according to formulas \eqref{inversion_bis} $A_1+A_2$ (resp. $A_1-A_2$) is invertible and its inverse is $(a_1+a_2)$ (resp. $a_1-a_2$). Thus :
\begin{equation}\label{formula_1}
U_1^{1,\infty}+U_2^{1,\infty}=2(A_1(\xi_1)+A_2(\xi_1))^{-1} \frac{m}{6\pi R} g  + \frac{1}{6\pi}(A_1(\xi_1)+A_2(\xi_1))^{-1} O(\sqrt{R})\underset{\alpha=1,2}{\underset{1 \leq i \leq N}{\max}} |U_\alpha^i|.
\end{equation}
We use the fact that $\|(A_1(\xi_1)+A_2(\xi_1))^{-1} \|$ is uniformly bounded independently of the particles and $N$ to get
$$
U_1^{1,\infty}+U_2^{1,\infty}=2(A_1(\xi_1)+A_2(\xi_1))^{-1} \frac{m}{6\pi R} g  +  O(\sqrt{R})\underset{\alpha=1,2}{\underset{1 \leq i \leq N}{\max}} |U_\alpha^i|.
$$
On the other hand, as $(U_1^{1,(0)},U_2^{1,(0)})=(U_1^1,U_2^1)$ we rewrite formula \eqref{formula_1} as :
\begin{equation*}
U_1^1+U_2^1 = - \underset{p=1}{\overset{\infty}{\sum}}(U_1^{1,(p)}+U_2^{1,(p)}) +(A_1(\xi_1)+A_2(\xi_1))^{-1} \frac{m}{6\pi R} g  \\+ 2\frac{1}{6\pi}(A_1(\xi_1)+A_2(\xi_1))^{-1} O(\sqrt{R})\underset{\alpha=1,2}{\underset{1 \leq i \leq N}{\max}} |U_\alpha^i|.
\end{equation*} 
Using again formula \eqref{reflection} this yields : 
\begin{align*}
U_1^1+U_2^1 &=  \underset{p=1}{\overset{\infty}{\sum}} \underset{j \neq1}{\sum}U[U_1^{j,(p-1)},U_2^{j,(p-1)}](x_1^1)+U[U_1^{j,(p-1)},U_2^{j,(p-1)}](x_2^1),\\
&  + 2 (A_1(\xi_1)+A_2(\xi_1))^{-1} \frac{m}{6\pi R} g  + \frac{1}{6\pi}(A_1(\xi_1)+A_2(\xi_1))^{-1} O(\sqrt{R})\underset{\alpha=1,2}{\underset{1 \leq i \leq N}{\max}} |U_\alpha^i|,\\
&=  \underset{j \neq1}{\sum} \left ( U[U_1^{j,\infty},U_2^{j,\infty}](x_1^1)+U[U_1^{j,\infty},U_2^{j,\infty}](x_2^1) \right )  +2(A_1(\xi_1)+ A_2(\xi_1))^{-1} \frac{m}{6\pi R} g, \\
& + \frac{1}{6\pi}(A_1(\xi_1)+A_2(\xi_1))^{-1} O(\sqrt{R})\underset{\alpha=1,2}{\underset{1 \leq i \leq N}{\max}} |U_\alpha^i|.
\end{align*}
We conclude by recalling that $\|(A_1+A_2)^{-1}\|_\infty=\| \mathbb{A}^{-1}\|_\infty $ is uniformly bounded.
\end{proof}
Applying the same ideas we obtain the following result:
\begin{prpstn}\label{velocity_mr}
for all $1\leq i \leq N$ we have :
$$
U_1^{i}-U_2^{i} =\underset{j\neq i}{\sum} \left ( U[U_1^{j,\infty},U_2^{j,\infty}](x_1^i) - U[U_1^{j,\infty},U_2^{j,\infty}](x_2^i) \right) +  O ( \sqrt{R}  |x_-^i| ) \underset{\alpha=1,2}{\underset{1 \leq i \leq N}{\max}} |U_\alpha^i|.\\
$$
$$
U_1^{i,\infty}-U_2^{i,\infty} = O ( \sqrt{R}|x_-^i| )\underset{\alpha=1,2}{\underset{1 \leq i \leq N}{\max}} |U_\alpha^i|.
$$
\end{prpstn}
\begin{proof}
The proof is analogous to the one of Proposition \ref{velocity_mr_bis}. The idea is to consider this time $w$ the unique solution to the Stokes equation \eqref{eq_stokes} completed with the following boundary conditions : 
\begin{equation}
w=\left\{ 
\begin{array}{rl}
W& \text {on $B(x_1^1, R),$}\\
-W& \text {on $B(x_2^1, R),$}\\
0 & \text {on $ B(x_1^i, R)\cup B(x_2^i, R)$, $i\neq 1$,}
\end{array}
\right.
\end{equation}
with $W$ an arbitrary vector of $\mathbb{R}^3$. Using the method of reflections, Propositions \ref{prop4} and \ref{convergence} we obtain the desired result.
\end{proof}
\subsection{Estimates for $\dot{x}_+^i$}
Propositions  \ref{velocity_mr_bis} and \ref{velocity_mr} yields the following result:
\begin{crllr}\label{velocity_+}
For all $1 \leq i \leq N$ we have : 
$$
U_+^i:= (\mathbb{A}(\xi_i))^{-1} \kappa g + \frac{6\pi r_0}{N} \underset{j\neq i}{\sum} \Phi(x_+^i-x_+^j)\kappa g + O(d_{\min}),
$$
where $\mathbb{A}= A_1+A_2$.
\end{crllr}
\begin{proof}
First of all, from Propositions \ref{velocity_mr_bis} and \ref{velocity_mr} we can show that the velocities $U_\alpha^i$ are uniformly bounded with respect to $N$ for all $ 1 \leq i \leq N$ and $\alpha=1,2$. Indeed, using formula \eqref{reflection} together with the decay properties \eqref{decay_rate} and Propositions \ref{velocity_mr_bis} and \ref{velocity_mr} we have : 
\begin{align*}
\underset{1 \leq i \leq N}{\underset{\alpha =1,2}{\max}} |U_\alpha^i| & \leq \underset{1\leq i \leq N}{\max}\,(|U_+^i|+|U_-^i|),\\
&\lesssim 1 + \underset{1\leq i \leq N}{\max}\, ( |U_+^{i,\infty}|+ |U_-^{i,\infty}|)+ O(\sqrt{R}) \underset{1 \leq i \leq N}{\underset{\alpha =1,2}{\max}} |U_\alpha^i|,\\
&\lesssim 1 + O(\sqrt{R)} \underset{1 \leq i \leq N}{\underset{\alpha =1,2}{\max}} |U_\alpha^i|.
\end{align*}
This allows us to bound the terms $\underset{1 \leq i \leq N}{\underset{\alpha =1,2}{\max}} |U_\alpha^i|$ by a constant independent of $N$ in the estimates of Propositions \ref{velocity_mr_bis} and \ref{velocity_mr}.
From Proposition \ref{velocity_mr} we have 
\begin{align*}
U_+^i &=  (A_1(\xi_i)+A_2(\xi_i))^{-1} \frac{m}{6 \pi R} g +\frac{1}{2} \underset{j \neq i}{\sum}\left ( U[U_1^{j,\infty},U_2^{j,\infty}](x_1^i)+U[U_1^{j,\infty}+U_2^{j,\infty}](x_2^i) \right ) + O(\sqrt{R}).
\end{align*}
We recall that using \eqref{def_formule} we have
\begin{align*}
U[U_1^{j,\infty},U_2^{j,\infty}](x_1^i)+U[U_1^{j,\infty},U_2^{j,\infty}](x_2^i)&= -(\Phi(x_+^j-x_1^i)+\Phi(x_+^j-x_2^i))(F_1^{j,\infty} +F_2^{j,\infty})\\
&+\mathcal{R}[U_1^{j,\infty},U_2^{j,\infty}](x_1^i)+\mathcal{R}[U_1^{j,\infty},U_2^{j,\infty}](x_2^i)
\end{align*}
recall that $ F_+^{j,\infty}  = - mg + O(R\sqrt{R}) $, see proof of Propositions \ref{velocity_mr} and \ref{velocity_mr_bis}. Hence, we replace $F_1^j+F_2^j=2 F_+^{j,\infty}$ by $-2mg$ with $6\pi R\kappa g = m g$ and bound the error terms using the decay properties of the Oseen tensor $\Phi$ and the field $\mathcal{R}$, see \eqref{decay_rate_R}
$$
 \underset{j \neq i}{\sum}\left |\Phi(x_+^j-x_1^i)+\Phi(x_+^j-x_2^i) \right|O(R \sqrt{R}) +\left|\mathcal{R}[U_1^{j,\infty},U_2^{j,\infty}](x_1^i)+\mathcal{R}[U_1^{j,\infty},U_2^{j,\infty}](x_2^i) \right| 
 \leq C \sqrt{R}+ C R.
$$
Now it remains to replace both terms $\Phi(x_+^j-x_1^i)$, $\Phi(x_+^j-x_2^i)$ by $\Phi(x_+^i-x_+^j)$. Direct computations show that for all $1 \leq \alpha \leq 2 $ we have 
$
| x_+^i -x_\alpha^i| =|x_-^i|,
$
which yields for all  $1 \leq \alpha,\leq 2 $:
$$
\underset{j \neq i }{\sum}|\Phi(x_+^j-x_\alpha^i) - \Phi(x_+^j -x_+^i) |6\pi R|\kappa g|\lesssim  \underset{j \neq i}{\sum} \frac{|x_-^i|R}{|x_+^1-x_+^j|^2}\lesssim |x_-^i|, 
$$
which is comparable to $R$ according to assumption \eqref{hyp_bound}. Gathering all the estimates, the error term is of order $\sqrt{R}$ which is of order $d_{\min}$ according to assumption \eqref{d_min+} and Remark \ref{rem1}.


\end{proof}
\subsection{Estimates for $\dot{x}_-^i$}
Analogously, Propositions  \ref{velocity_mr_bis} and \ref{velocity_mr} yields the following result:
\begin{crllr}\label{velocity_-}
For all $1 \leq i \neq N$ we have: 
$$
\frac{U_1^i-U_2^i}{2}=\left(  \frac{6\pi r_0}{N} \underset{j\neq i}{\sum}\nabla \Phi(x_+^i-x_+^j)\kappa g \right)\cdot x_-^i   + O \left( {|x_-^i|}d_{\min} \right).
$$
\end{crllr}
\begin{proof}
The first formula of Proposition \ref{velocity_mr} together with the uniform bound on the velocities $(U_+^i, U_-^i)$, see proof of Corollary \ref{velocity_+}, yields: 
\begin{align*}
U_1^i-U_2^i= \underset{j\neq i}{\sum} U[U_1^{j,\infty}, U_2^{j,\infty}](x_1^i) - U[U_1^{j,\infty}, U_2^{j,\infty}](x_2^i)+ O(\sqrt{R}|x_-^i|).
\end{align*}
We want to estimate the first term, we have using \eqref{def_formule} 
\begin{align*}
& U[U_1^{j,\infty}, U_2^{j,\infty}](x_1^i) - U[U_1^{j,\infty}, U_2^{j,\infty}](x_2^i)\\
& =- \Phi(x_1^i-x_+^j)(F_1^{j,\infty}+ F_2^{j,\infty}) + \Phi(x_2^i-x_+^j) (F_1^{j,\infty}+ F_2^{j,\infty}) ,\\
& = -4 [ \nabla \Phi(x_2^i-x_+^j)\cdot x_-^i] F_+^{j,\infty}\\
&-2 \int_0^1 \underset{|\beta|=2}{\sum} (x_-^i)^\beta \cdot D^\beta \Phi(x_2^i-x_+^j+tx_-^i)F_+^{j,\infty} dt\\
&=- 4 [\nabla \Phi(x_+^i-x_+^j) \cdot x_-^i ]F_+^{j,\infty}  + \mathcal{E}^1_{i,j}+\mathcal{E}^2_{i,j}. \\
\end{align*}
Now recall that, from the proof of Proposition \ref{velocity_mr_bis} we have:
$$
F_+^{j,\infty}=- mg + O(R^2).
$$
Thus, we get the following formula:
\begin{align*}
U[U_1^{j,\infty}, U_2^{j,\infty}](x_1^i) - U[U_1^{j,\infty}, U_2^{j,\infty}](x_2^i)&=  2 [\nabla \Phi(x_+^i-x_+^j) \cdot x_-^i ]mg + \mathcal{E}^1_{i,j}+\mathcal{E}^2_{i,j}+\mathcal{E}^3_{j},
\end{align*}
with
$$
\mathcal{E}^3_{j}=- 4 [\nabla \Phi(x_+^i-x_+^j) \cdot x_-^i ](F_+^{j,\infty}+ mg).
$$
We recall that $m g =6\pi R \kappa g = \frac{1}{2} \frac{6 \pi r_0}{N} g$. Finally we obtain: 
\begin{align*}
\frac{U_1^i-U_2^i}{2}&= \frac{6\pi r_0}{N} \underset{j\neq i}{\sum} [\nabla \Phi(x_+^i-x_+^j) \cdot x_-^i ]\kappa g  +  \frac{1}{2}\underset{j\neq i}{\sum} \mathcal{E}^1_{i,j}+\mathcal{E}^2_{i,j}+\mathcal{E}^3_{j} + O(\sqrt{R}|x_-^i|).
\end{align*}
It remains to bound the error terms. We begin by the first one: 
\begin{align*}
|\mathcal{E}_{i,j}^1 |& \leq 2 \left ( \underset{y \in[x_1^i,x_2^i]}{\sup} \left (| \nabla^2 \Phi(x_+^j-y)|\right) \right) |x_-^i|^2 (|F_+^{j,\infty}|).
\end{align*}
We emphasize that  for all $y\in [x_1^i,x_2^i]$:
$$
|y-x_+^j|\geq |x_1^i- x_+^j|-|x_1^i- y|
\geq |x_1^i-x_+^j| - 2|x_-^i|
\geq \frac{1}{4}|x_+^i-x_+^j|,
$$
where we used the fact that $$|x_-^i| \leq \frac{C}{R} \leq \frac{1}{8}d_{\min} \leq \frac{1}{8} |x_+^i-x_+^j|,$$
and 
$$
|x_1^i-x_+^j| \geq \frac{1}{2} |x_+^i-x_+^j|,
$$
This yields : 
$$
\underset{ j\neq i}{\sum} |\mathcal{E}_{i,j}^1|  \leq C \underset{ j\neq i}{\sum} \frac{1}{d_{ij}^3} |x_-^i|^2 R \kappa |g| \leq C{|x_-^i|}\frac{R}{d_{\min}} \left( \underset{ j\neq i}{\sum} \frac{R}{d_{ij}^2}\right) \leq C{|x_-^i|}\frac{R}{d_{\min}}\leq C d_{\min}|x_-^i|.
$$
For the second error term we have:
\begin{equation*}
\mathcal{E}_{i,j}^2 =- 2 [ \nabla \Phi(x_2^i-x_+^j)- \nabla \Phi(x_+^i-x_+^j)]\cdot x_-^i \cdot F_+^{j,\infty} ,
\end{equation*}
where
\begin{align*}
| \nabla \Phi(x_2^i-x_+^j)- \nabla \Phi(x_+^i-x_+^j)|\leq C \left (\frac{1}{|x_2^i-x_+^j|^3}+\frac{1}{|x_+^i-x_+^j|^3}  \right) | x_-^i|.
\end{align*}
Since $|x_-^j| \sim R \sim |x_-^i|$ the second error term is bounded by: 
\begin{align*}
\underset{ j\neq i}{\sum} |\mathcal{E}_{i,j}^2| &\leq C 
\underset{ j\neq i}{\sum} \frac{1}{d_{ij}^3} |x_-^i|^2 R \kappa |g|,
\end{align*}
which yields the same estimate as for the first error term.
Finally, the last error term gives: 
$$
\underset{ j\neq i}{\sum} |\mathcal{E}_{i,j}^3|  \leq 2 | \nabla \Phi(x_+^i-x_+^j)| \,| x_-^i |\,|F_+^{j,\infty}+ mg| \leq C R^2.
$$
where we used the fact that 
$
F_+^{j,\infty} = -mg + O(R^2)
$
and $|x_-^i| \sim R$.
\end{proof}
\section{Proof of Theorem \ref{thm}}
In order to derive the transport-Stokes equation satisfied at the limit, the idea is to show that the discrete density $\mu^N$ satisfies weakly a transport equation. We introduce the following notations.  Given a density $\rho$, we define the operator $\mathcal{K} \rho$ as: 
$$
\mathcal{K}\rho(x) := 6\pi r_0 \int_{\mathbb{R}^3} \Phi(x-y) \kappa g \, \rho(dy).
$$
The operator is well defined and is Lipschitz in the case where $\rho \in L^1 \cap L^\infty$. Moreover, note that $\mathcal{K} \rho$ satisfies the Stokes equation 
$$ - \Delta \mathcal{K}(\rho) + \nabla p = 6\pi r_0 \kappa g \rho, $$ on $\mathbb{R}^3$. Analogously, we define $\mathcal{K}^N \rho ^N$ as: 
\begin{equation}\label{def_K^N}
\mathcal{K}^N\rho^N(x):= 6\pi r_0 \int_{\mathbb{R}^3} \chi \Phi(x-y) \kappa g \,\rho^N(dy),
\end{equation}
where $\chi \Phi (\cdot)= \chi\left (\frac{\cdot}{d_{\min}} \right) \Phi(\cdot)$, $\chi$ is a truncation function such that $\chi=0 $ on $B(0,1/4)$ and $\chi=1$ on ${}^c B(0,1/2)$.
\subsection{Derivation of the transport-Stokes equation}
The transport equation satisfied by $\mu^N$ is obtained directly using the ODE system derived for each couple $(x_+^i$, $\xi_i)$. We recall that:
\begin{eqnarray*}
U_+^i & = & (\mathbb{A}(\xi_i))^{-1} \kappa g +  \mathcal{K}^N \rho^N(x_+^i) +O(d_{\min}), \\
\frac{U_-^i}{R}& =&  \nabla \mathcal{K}^N \rho^N(x_+^i) \cdot \xi_i  + O\left(d_{\min}\right).
\end{eqnarray*}
Following the idea of \cite[Section 5.2]{Mecherbet} and using the fact that the centers $x_+^i$ do not collide thanks to our assumptions, one can show that we can construct two divergence-free velocity fields $E^N$ and $\tilde{E}^N$ such that  : 
\begin{eqnarray}\label{EDO_N}
U_+^i & = & (\mathbb{A}(\xi_i))^{-1} \kappa g +  \mathcal{K}^N \rho^N(x_+^i) + E^N(x_+^i),  \\
\frac{U_-^i}{R}& =&  \nabla \mathcal{K}^N \rho^N(x_+^i) \cdot \xi_i  + \tilde{E}^N\left(x_+^i\right), \nonumber
\end{eqnarray}
and there exists a positive constant independent of $N$ such that
\begin{eqnarray}\label{bound_err_uniform}
\|E^N\|_\infty = O(d_{\min}),& \|\tilde{E}^N\|_\infty= O\left( d_{\min} \right),&\|\nabla E^N\|_\infty +\|\nabla \tilde{E}^N\|_\infty < C.
 \end{eqnarray}
This construction yields the following result
\begin{prpstn}\label{VS_N}
$\mu^N$ satisfies weakly the transport equation:
\begin{equation}\label{ff_VS_N}
\partial_t \mu^N + \div_x [ (\mathbb{A}(\xi))^{-1} \kappa g \mu^N +  \mathcal{K}^N \rho^N(x) \mu^N+E^N \mu^N]+\div_\xi [ \nabla \mathcal{K}^N \rho^N(x) \cdot \xi \mu^N+ \tilde{E}^N(x) \mu^N] =0.
\end{equation}
\end{prpstn}
We can prove now Theorem \ref{thm}.
\subsection{proof of Theorem \ref{thm} }
The proof is a corollary of Proposition \ref{VS_N}. Indeed, we want to show that for all $\psi \in \mathcal{C}_c^\infty(\mathbb{R}^3)$ we have:
\begin{multline}
\int_0^T\int_{\mathbb{R}^3\times \mathbb{R}^3} \Big\{\partial_t \psi(t,x,\xi)+ \nabla_x \psi(t,x,\xi) \cdot [(\mathbb{A}(\xi))^{-1} \kappa g +  \mathcal{K} \rho(x)))]\\
+\nabla_\xi \psi(t,x,\xi) \cdot [ \nabla \mathcal{K} \rho(x) \cdot \xi ] \Big\}\mu(t,dx,d\xi) dt=0.
\end{multline}
which is obtained directly by passing through the limit in each term of formula \eqref{ff_VS_N}. Indeed we recall that we have the following estimates:
\begin{eqnarray*}
\|\mathcal{K}^N\rho^N - \mathcal{K} \rho \|_\infty &\lesssim& W_\infty, \\
\|\nabla \mathcal{K}^N\rho^N - \nabla \mathcal{K} \rho \|_\infty &\lesssim& W_\infty(1+ |\log W_\infty|) ,\\
\|E^N\|_\infty = O\left(d_{\min} \right)&, & \|\tilde{E}^N\|_\infty = O\left(d_{\min} \right).
\end{eqnarray*}
\section{Proof of theorem \ref{thm_bis} and \ref{thm2}}
This section is devoted to the proof of Theorem \ref{thm_bis} and \ref{thm2}. The Lipschitz-like estimates proved in Proposition \ref{conservation_distance} suggests a correlation between the vectors along the line of centers $\xi_i$ and the centers $x_+^i$. In this section, we show in particular that this correlation is well propagated in time. 
\subsection{Derivation of the transport-Stokes equation}
We assume now that there exists a lipschitz function $F_0$ such that 
$$
\xi_i(0)=F_0(x_+^i(0))\,,\:\: 1 \leq i \leq N,
$$
which means that $\mu^N_0= \rho^N_0 \otimes \delta_{F_0}$. In order to propagate this correlation we search for a function $F^N(t,\cdot)\in W^{1,\infty}(\mathbb{R}^3)$ such that for all $t\in[0,T]$ we have 
$$
\xi_i(t)=F^N(t,x_+^i(t))\,,\:\: 1 \leq i \leq N.
$$
According to the ODE satisfied by $\xi_i$, see \eqref{EDO_N}, $F^N$ must satisfy the following equation
\begin{equation*}
\left\{
\begin{array}{rcl}
\partial_t F^N+ \nabla F^N \cdot (\mathbb{A}(F^N)^{-1}\kappa g+  \mathcal{K}^N\rho^N+E^N) &=& \nabla \mathcal{K}^N\rho^N \cdot F^N + \tilde{E}^N,\\
F^N(0,\cdot) & =& F_0.
\end{array}
\right.
\end{equation*}
The following proposition shows the existence and uniqueness of $F^N$.
\begin{prpstn}\label{existence_FN}
There exists $T>$0 such that for all $N\in \mathbb{N}^*$, there exists a unique (local) solution $F^N \in L^\infty(0,T;W^{1,\infty}(\mathbb{R}^3))$ of the following equation
\begin{equation}\label{correlationN}
\left\{
\begin{array}{rcl}
\partial_t F^N+ \nabla F^N \cdot (\mathbb{A}(F^N)^{-1}\kappa g+  \mathcal{K}^N\rho^N+E^N) &=& \nabla \mathcal{K}^N\rho^N \cdot F^N + \tilde{E}^N, \\
F^N(0,\cdot) & =& F_0.
\end{array}
\right.
\end{equation}
\end{prpstn}
\begin{proof}
The idea is to apply a fixed-point argument. 
We define the mapping $ \mathcal{A}$ which associates to any $F \in L^\infty(0,T;W^{1,\infty}(\mathbb{R}^3)) $ the unique solution $\mathcal{A}(F)=\hat{F}$ to the transport equation
\begin{equation}\label{transport_pf}
\left\{
\begin{array}{rcl}
\partial_t \hat{F}+ \nabla \hat F \cdot (\mathbb{A}(F)^{-1}\kappa g+  \mathcal{K}^N\rho^N+E^N) &=& \nabla \mathcal{K}^N\rho^N \cdot F + \tilde{E}^N,\\
\hat{F}(0,\cdot) & =& F_0.
\end{array}
\right.
\end{equation}
We define $X^N$ as the characteristic flow satisfying :
$$
\partial_s X^N(s,t,x) = \mathbb{A}(F(s,X^N(s,t,x)))^{-1}\kappa g+  \mathcal{K}^N\rho^N(s,X^N(s,t,x))+E^N(s,X^N(s,t,x)).
$$
$$X^N(t,t,x)  = x.$$
The Lipschitz property of $\mathbb{A}^{-1}$, $F$, $\mathcal{K}^N\rho^N$ and $E^N$ ensures the existence, uniqueness and regularity of such a flow, see Proposition \ref{boundedness} and formula \eqref{bound_err_uniform}. Moreover, direct estimates show that for all $0\leq s\leq t$:
\begin{multline}\label{lipschitz_X}
\|\nabla X^N(s,t,\cdot)\|_\infty \leq \\ \text{exp}(\left[|\kappa g| \| \nabla \mathbb{A}^{-1}\|_\infty \| F\|_{L^\infty(0,T;W^{1,\infty}(\mathbb{R}^3))} +\| \mathcal{K}^N\rho^N+E^N\|_{L^\infty(0,T;W^{1,\infty}(\mathbb{R}^3))} \right](t-s)).
\end{multline}
Hence, we can write
\begin{multline*}
\hat{F}(t,x)=F_0(X^N(0,t,x))+ \int_0^t \nabla \mathcal{K}^N\rho^N(s,X^N(s,t,x)) \cdot F(s,X^N(s,t,x))\\+\tilde{E}(s,X^N(s,t,x)) ds.
\end{multline*}
Direct computations yield 
\begin{multline*}
\|\mathcal{A}(F)\|_{L^\infty(0,T;L^\infty(\mathbb{R}^3))} \leq \|F_0\|_\infty+T \|\nabla \mathcal{K}^N\rho^N\|_{L^\infty(0,T; L^\infty(\mathbb{R}^3))} \|F\|_{L^\infty(0,T;L^\infty(\mathbb{R}^3))} +\\ \| \tilde{E}^N\|_{L^\infty(0,T;L^\infty(\mathbb{R}^3))} ,
\end{multline*}
and 
\begin{multline*}
\| \nabla \mathcal{A}({F})\|_{L^\infty(0,T; L^\infty(\mathbb{R}^3))} \leq  [ \|F_0\|_{1,\infty} + T \|\tilde{E}^N\|_{L^\infty(0,T;W^{1,\infty}(\mathbb{R}^3))} + \\T  \Big\{ \|\nabla \mathcal{K}^N \rho^N\|_{L^\infty(0,T;W^{1,\infty}(\mathbb{R}^3))} \Big \}\|F\|_{L^\infty(0,T; W^{1,\infty}(\mathbb{R}^3))}] \| \nabla X^N(\cdot,t,\cdot)\|_{L^\infty(0,T;L^\infty(\mathbb{R}^3))},
\end{multline*}
Gathering all the estimates and using Proposition \ref{boundedness} and the uniform bounds \eqref{bound_err_uniform}, there exists some constants independent of $N$ such that: 
\begin{equation}\label{bound}
\|\mathcal{A}(F)\|_{L^\infty(0,T;W^{1,\infty}(\mathbb{R}^3))} \leq (\|F_0\|_{W^{1,\infty}(\mathbb{R}^3)}+ T C_1 +TC_2 \|F\|_{L^\infty(0,T;W^{1,\infty}(\mathbb{R}^3))} ) e^{C_3 T}.
\end{equation}
On the other hand, given $F_1$, $F_2 \in L^\infty(0,T;W^{1,\infty}(\mathbb{R}^3))$ we set $X_i$ the associated characteristic flow and we have
\begin{multline*}
\|\mathcal{A}(F_1)(t,\cdot)- \mathcal{A}(F_2)(t,\cdot)\|_\infty  \leq \\ \Big (\|\nabla F_0\|_\infty+ t \|F_1\|_{L^\infty(0,T;W^{1,\infty}(\mathbb{R}^3))} \| \mathcal{K}^N\rho^N \|_{L^\infty(0,T;W^{2,\infty}(\mathbb{R}^3))}+  t \|\tilde{E}^N\|_{L^\infty(0,T;W^{1,\infty}(\mathbb{R}^3))}  \Big)\\
\times \|X_1(0,t,\cdot)-X_2(0,t,\cdot)\|_\infty \\
+t\|\nabla \mathcal{K}^N\rho^N \|_{L^\infty(0,T;L^\infty(\mathbb{R}^3))}  \|F_1-F_2\|_{L^\infty(0,T; L^\infty(\mathbb{R}^3))}.
\end{multline*}
The characteristic flows satisfies
\begin{multline*}
|X_1(s,t,x)-X_2(s,t,x)| \leq\|\nabla \mathbb{A}^{-1}\|_\infty  \int_s^t  \|F_1(\tau,\cdot)-F_2(\tau,\cdot)\|_\infty + \\ ( \|F_1\|_{L^\infty(0,T;L^\infty)} |\kappa g| + 2 \|\nabla \mathcal{K}^N \rho^N+\nabla E^N\|_{L^\infty(0,T;L^\infty)} ) | X_1(\tau,t,x) -X_2(\tau,t,x) | d \tau,
\end{multline*}
hence
$$
\|X_1(s,t,\cdot)-X_2(s,t,\cdot)\|_\infty \leq \left(\int_s^t \|\nabla \mathbb{A}^{-1}\|_\infty \|F_1(\tau,\cdot)-F_2(\tau,\cdot)\|_\infty d\tau\right) e^{C(t-s)}. $$
This yields
\begin{equation}\label{bound1}
\|\mathcal{A}(F_1) -\mathcal{A}(F_2)\|_{L^\infty(0,T; L^\infty(\mathbb{R}^3))} \leq C(\|F_1\|_{L^\infty(0,T;W^{1,\infty}(\mathbb{R}^3))})\, T\, \|F_1-F_2\|_{L^\infty(0,T; L^\infty(\mathbb{R}^3))}.
\end{equation}
We construct the following sequence $({F}_k)_{k\in \mathbb{N}} \subset L^\infty(0,T;W^{1,\infty}(\mathbb{R}^3)) $ defined as 
\begin{equation*}
\left\{
\begin{array}{rcl}
F^{k+1}& = &\mathcal{A}(F^k) \,, k \in \mathbb{N}\,,\\
F^0&=&F_0\,.
\end{array}
\right.
\end{equation*}
For $T$ small enough and independent of $N$, using estimates \eqref{bound} and \eqref{bound1}, the sequence $(F^k)_k$ is bounded in $L^\infty(0,T;W^{1,\infty}(\mathbb{R}^3))$ and is a Cauchy sequence in the Banach space $L^\infty(0,T; L^\infty(\mathbb{R}^3))$. There exists a limit $F \in L^\infty(0,T;W^{1,\infty}(\mathbb{R}^3))$ such that $F^k\to F$ in $L^\infty(0,T,L^\infty(\mathbb{R}^3))$ and $\nabla F^k \rightharpoonup \nabla F$ weakly-* in $L^\infty(0,T,L^\infty(\mathbb{R}^3))$. It remains to show that $F=\mathcal{A}({F})$. The weak formulation of the transport equation writes 
\begin{multline*}
\int_0^T \int_{\mathbb{R}^3} \left(\partial_t \psi + \div\left( \psi \cdot [\mathbb{A}^{-1}({F}^k) \kappa g + \mathcal{K}^N\rho^N]+ E^N\right) \right){F}^k =\\ \int_0^T \int_{\mathbb{R}^3}\left( \nabla \mathcal{K}^N\rho^N \cdot {F}^k+\tilde{E}^N \right) \cdot \psi,
\end{multline*}
for all $\psi \in \mathcal{C}_c^1((0,T) \times \mathbb{R}^3)$. Using the strong convergence of $F^k$ to $F$ we get
\begin{multline*}
\int_0^T \int_{\mathbb{R}^3} \left(\partial_t \psi + \div\left( \psi \cdot [\mathbb{A}^{-1}({F}) \kappa g + \mathcal{K}^N\rho^N]+ E^N\right) \right){F} =\\ \int_0^T \int_{\mathbb{R}^3}\left( \nabla \mathcal{K}^N\rho^N \cdot {F}+ \tilde{E}^N  \right)\cdot \psi,
\end{multline*}
Uniqueness of the fixed-point is ensured thanks to estimate \eqref{bound1}.
\end{proof}
Proposition \ref{existence_FN} and formula \eqref{EDO_N} yield the following result
\begin{crllr}
There exists a unique solution of \eqref{correlationN} $F^N\in L^\infty(0,T; W^{1,\infty}(\mathbb{R}^3))$ such that $\mu^N= (\text{id}, F^N) \# \rho^N$ and $\rho^N$ satisfies weakly 
\begin{equation}\label{rho^N_transport}
\partial_t \rho^N + \div [ (\mathbb{A}(F^N))^{-1} \kappa g  +  \mathcal{K}^N \rho^N(x)+E^N) \rho^N] =0.
\end{equation}
\end{crllr}
\subsection{proof of Theorem \ref{thm_bis} and \ref{thm2}}
In the previous part we showed the existence of a unique function $F^N$ such that:
$$
\xi_i=F^N(x_+^i).
$$
In order to provide the limit behaviour of the system, we need to extract the limit equation satisfied by $F= \underset{N\to \infty}{\lim} F^N$ and to estimate and specify the convergence. It is straightforward that the limit function $F$ should satisfy the following equation:
\begin{equation}\label{correlation}
\left\{
\begin{array}{rcl}
\partial_t F+ \nabla F \cdot (\mathbb{A}(F)^{-1}\kappa g+  \mathcal{K}\rho) &=& \nabla \mathcal{K}\rho \cdot F, \text{ on $[0,T] \times \mathbb{R}^3,$}\\
F(0,\cdot) & =& F_0.
\end{array}
\right.
\end{equation}
We begin with the proof of local existence and uniqueness of the solution to system \eqref{eq_limite}.
\begin{proof}[Proof of Theorem \ref{thm2}]
Let $p>3$, $F_0 \in W^{2,p}(\mathbb{R}^3)$, $\rho_0 \in W^{1,p}(\mathbb{R}^3)$ having compact support. The idea is to apply a fixed-point argument. We define the operator $A$ which associates to each $u \in  L^\infty(0,T;W^{3,p}(\mathbb{R}^3))$ the following divergence free velocity
$$
 u\mapsto F(u) \mapsto \rho(u) \mapsto \mathcal{A}(u),
$$
where $F(u)\in L^\infty(0,T;W^{2,p}(\mathbb{R}^3)) $ is the unique solution, see Proposition \ref{stab_F}, to the following equation 
\begin{equation*}
\left\{
\begin{array}{rcll}
\partial_t F+ \nabla F \cdot ( \mathbb{A}^{-1}(F) \kappa g + u ) &=& \nabla u \cdot F,\,& \text{ on $[0,T] \times \mathbb{R}^3$},\\
F(0,\cdot) &=& F_0, \,& \text{ on $ \mathbb{R}^3$.}
\end{array}
\right.
\end{equation*}
$\rho(u)\in L^\infty(0,T;W^{1,p}(\mathbb{R}^3)) $ is the unique solution, see Proposition \ref{stab_rho}, to the transport equation 
\begin{equation*}
\left\{
\begin{array}{rcll}
\partial_t \rho+ \div  ( (\mathbb{A}^{-1}(F(u)) \kappa g + u) \rho ) &=& 0,\,& \text{ on $[0,T] \times \mathbb{R}^3$},\\
\rho(0,\cdot) &=& \rho_0, \,& \text{ on $ \mathbb{R}^3$.}
\end{array}
\right.
\end{equation*}
and $\mathcal{A}(u) = \mathcal{K} \rho(u)=6\pi r_0 \Phi * (\kappa \rho(u) g)$. 
The mapping is well-defined, indeed, since $\rho_0 \in W^{1,p}(\mathbb{R}^3)$ we have $\rho \in L^\infty(0,T; W^{1,p}(\mathbb{R}^3))$, see Proposition \ref{stab_rho}. Consequently, applying \cite[Theorem IV.2.1]{Galdi} shows that $\nabla^3 A(u)$, $\nabla^2 A(u) \in L^p(\mathbb{R}^3)$ and we have 
\begin{eqnarray*}
\|\nabla^3 A(u)\|_p  \leq C \|\nabla \rho(u)\|_p, & \|\nabla^2 A(u)\|_p \leq C \|\rho(u)\|_p.
\end{eqnarray*}
On the other hand, since $\rho(t,\cdot) \in L^p(\mathbb{R}^3)$ and is compactly supported, see Remark \ref{rem:supp_comp}, we have in particular $\rho (t,\cdot) \in L^{q_1}(\mathbb{R}^3) \cap L^{q_2}(\mathbb{R}^3)$ with 
\begin{eqnarray*}
\displaystyle q_1=\frac{3p}{3+p} \in ]3/2,3[,& q_2=\dfrac{3p}{3+2p} \in ]1,3/2[.
\end{eqnarray*}
We apply again \cite[Theorem IV.2.1]{Galdi} for $q=q_1$ (resp. $q=q_2$) to get $\nabla A(u) \in L^p(\mathbb{R}^3)$ (resp. $A(u)\in L^p(\mathbb{R}^3)$) and we have according to \cite[Formula IV.2.22]{Galdi} (resp. \cite[Formula IV.2.23]{Galdi})
\begin{eqnarray*}
\|\nabla A(u)\|_p \leq C \|\rho(u)\|_{q_1},& \|A(u)\|_p \leq C\|\rho(u)\|_{q_2}&, 
\end{eqnarray*}
Hence, since $q_1,q_2 <3<p$, Holder's inequality yields
$$
\|\nabla A(u)\|_p+\|A(u)\|_p   \lesssim( \underset{[0,T]}{\sup} | \supp \rho(u)(t,\cdot)|^{1/3}+\underset{[0,T]}{\sup} | \supp \rho(u)(t,\cdot)|^{2/3}) \|\rho(u)\|_p ,
$$
where $\underset{[0,T]}{\sup} | \supp \rho(u)(t,\cdot)|$ depends on $T$, $\|\mathbb{A}^{-1}\|_\infty$, $\|F\|_{L^\infty(0,T;W^{2,p}(\mathbb{R}^3))}$ and $\|u\|_{L^\infty(0,T;W^{2,p}(\mathbb{R}^3))}$ according to Remark \ref{rem:supp_comp}
\begin{equation}\label{borneu_support}
\diam(\supp (\rho(u)(t,\cdot)) 
\leq C(\rho_0, T, \|u\|_{L^\infty(0,T;W^{2,p}(\mathbb{R}^3))}, \|F\|_{L^\infty(0,T;W^{2,p}(\mathbb{R}^3))}),
\end{equation}
Finally we have 
\begin{eqnarray}
\|A(u)\|_{L^\infty(0,T;W^{3,p}(\mathbb{R}^3))} &\leq & C(1+M(T)) \|\rho(u)\|_{L^\infty(0,T;W^{1,p}(\mathbb{R}^3))}, \label{borne_u}\\
\|A(u)\|_{L^\infty(0,T;W^{2,p}(\mathbb{R}^3))} &\leq & C(1+M(T)) \|\rho(u)\|_{L^\infty(0,T;L^p(\mathbb{R}^3))}, \label{stab_u}
\end{eqnarray}
$$M(T)=\underset{[0,T]}{\sup} | \supp \rho(u)(t,\cdot)|^{1/3}(1+\underset{[0,T]}{\sup} | \supp \rho(u)(t,\cdot)|^{1/3}).$$
We recall the following bounds, see Proposition \ref{stab_rho} and Proposition \ref{stab_F}
\begin{eqnarray}\label{borne_rho}
\|\rho(u)\|_{L^\infty(0,T;W^{1,p}(\mathbb{R}^3))}  \leq  \|\rho_0\|_{1,p}e^{CT},
\end{eqnarray}
with $ C=C(\|F(u)\|_{L^\infty(0,T; W^{2,p}(\mathbb{R}^3))},\|u\|_{L^\infty(0,T; W^{3,p}(\mathbb{R}^3))})$. According to Proposition \ref{stab_F}, for a small time interval we have for a fixed $\lambda>1$
\begin{equation}\label{borne_F}
\|F(u)\|_{2,p} \leq \lambda \|F_0\|_{2,p}.
\end{equation}
On the other hand, gathering the stability estimates of Proposition \ref{stab_rho} and Proposition \ref{stab_F} and \eqref{stab_u} we get for $u_i \in W^{3,p}(\mathbb{R}^3)$, $i=1,2$
\begin{align*}
&\|A(u_1)-A(u_2)\|_{L^\infty(0,T;W^{2,p}(\mathbb{R}^3))} \\
&\leq C (1+M(u_1,u_2)(T)) \| \rho(u_1)-\rho(u_2)\|_{L^\infty(0,T;L^p(\mathbb{R}^3))} \\
& \leq C (1+M(u_1,u_2)(T)) T \left(\|F(u_1)-F(u_2)\|_{L^\infty(0,T;W^{1,p}(\mathbb{R}^3))}+ \|u_1-u_2\|_{L^\infty(0,T;W^{1,p}(\mathbb{R}^3))} \right) e^{C_1 T} \\
& \leq C  (1+M(u_1,u_2)(T)) T(1+T)\|u_1-u_2\|_{L^\infty(0,T;W^{2,p}(\mathbb{R}^3))}e^{C_1 T},
\end{align*}
where $C$ depends on $\|u_i\|_{L^\infty(0,T;W^{3,p}(\mathbb{R}^3))}$, $\|F(u_i)\|_{L^\infty(0,T;W^{2,p}(\mathbb{R}^3))}$, $ \|\rho(u_i)\|_{L^\infty(0,T;W^{1,p}(\mathbb{R}^3))} $ and 
\begin{align*}
M(u_1,u_2)(T)&:= \underset{[0,T]}{\sup} |\supp (\rho(u_1)) \cup \supp(\rho(u_2) |^{1/3}(1+\underset{[0,T]}{\sup} |\supp (\rho(u_1)) \cup \supp(\rho(u_2) |^{1/3}),\\
&\lesssim C(T,\|u_i\|_{L^\infty(0,T;W^{2,p}(\mathbb{R}^3))},\|F_i\|_{L^\infty(0,T;W^{2,p}(\mathbb{R}^3))}, \supp(\rho_0)).
\end{align*}
We consider the following sequence 
\begin{equation*}
\left\{
\begin{array}{rcl}
u^{k+1}& = &\mathcal{A}(u^k) \,, k \in \mathbb{N}\,,\\
u^0&=&\mathcal{K}\rho_0 \,.
\end{array}
\right.
\end{equation*}
We set $F^k:=A(u^k)$, $\rho^k:=\rho(u^k)$. Previous estimates show that the sequences $(u_k)_{k \in \mathbb{N}}$, $(F_k)_{k \in \mathbb{N}}$, $(\rho_k)_{k \in \mathbb{N}}$ are uniformly bounded in $L^\infty(0,T;W^{3,p}(\mathbb{R}^3))$, $L^\infty(0,T;W^{2,p}(\mathbb{R}^3))$,\\ $L^\infty(0,T;W^{1,p}(\mathbb{R}^3))$, respectively, and are Cauchy sequences in $ L^\infty(0,T;W^{2,p}(\mathbb{R}^3))$,\\ $L^\infty(0,T;W^{1,p}(\mathbb{R}^3))$, $L^\infty(0,T;L^p(\mathbb{R}^3))$, respectively for $T$ small enough. Consequently, there exists $(u,F,\rho)$ such that
\begin{eqnarray*}
u^k &\to u &  \text{ in } L^\infty(0,T; W^{2,p}(\mathbb{R}^3)),\\
F^k &\to F &  \text{ in } L^\infty(0,T; W^{1,p}(\mathbb{R}^3)),\\
\rho^k &\to \rho &  \text{ in } L^\infty(0,T;L^p(\mathbb{R}^3)).
\end{eqnarray*}
This allows to pass through the limit in the weak formulations of $u^k$ and $\rho^k$. In addition, we use the fact that $\nabla F_k$ converges weakly-* in $L^\infty(0,T; L^\infty(\mathbb{R}^3))$ in order to pass through the limit in the weak formulation of $F^k$.
Hence, the triplet $(u,\rho,F)$ satisfies equation \eqref{eq_limite}. We recover the regularity of each term using the a priori bounds. Uniqueness is a consequence of the previous stability estimates.
\end{proof}
\subsection{Proof of Theorem \ref{thm_bis}}
\begin{proof}[Proof of Theorem \ref{thm_bis}]
We recall that $W_\infty(\rho^N,\rho) \to 0 $ according to \eqref{W0}. We want to show that the triplet $(\rho^N, F^N, \mathcal{K}^N\rho^N)$ converges to 
$(\rho, F, \mathcal{K}\rho)$ the unique solution of equation \eqref{eq_limite}. From Proposition \ref{inifnite_cv} and using the same arguments as in Proposition \ref{stab_F} we have
$$
\|F^N(t,\cdot)-F(t,\cdot)\|_\infty \leq C \int_0^t  W_\infty(s) \left (1+ |\log W_\infty(s)|)+ \frac{W_\infty^2(s)}{d_{\min}^2}  \right ) + \|E^N\|_\infty + \|\tilde{E}^N\|_\infty,
$$
where $W_\infty(s):= W_\infty(\rho^N(s,\cdot),\rho(s,\cdot))$. Hence $F^N $ converges to $F$ in $L^\infty(0,T; L^\infty(\mathbb{R}^3))$ and $\mathcal{K}^N \rho^N$ converges to $\mathcal{K} \rho$ in $L^\infty(0,T; W^{1,\infty}(\mathbb{R}^3))$ if the Wasserstein distance is preserved in finite time. This allows us to pass through the limit in the weak formulation of $\rho^N$
$$
\int_0^t \int_{\mathbb{R}^3} \left( \partial_t \psi+\nabla \psi \cdot \left (\mathcal{A}^{-1}(F^N) \kappa g + \mathcal{K}^N\rho^N \right ) \right) \rho^N =0.
$$
\end{proof}

\appendix 
\section{Some preliminary estimates}
This section is devoted to the proof of the following lemma which is analogous to \cite[Lemma 2.1]{JO}. We drop the dependence with respect to time in what follows.
\begin{lmm}\label{JO}
There exists a positive constant $C$ such that for $k\in[0,2]$ and $N$ large enough
\begin{eqnarray*}
\frac{1}{N} \underset{j \neq i}{\sum} \frac{1}{d_{ij}^k} & \leq & C  \left(\|\rho\|_\infty \frac{W_\infty^3}{d_{\min}^k} + \|\rho\|_\infty^{k/3}\right)\,,\\
\frac{1}{N} \underset{j \neq i}{\sum} \frac{1}{d_{ij}^3} & \leq & C\|\rho\|_\infty \left( \frac{W_\infty^3}{d_{\min}^3} + | \log\left( \|\rho\|_\infty^{1/3} W_\infty\right)|+1 \right)\,.
\end{eqnarray*}
\end{lmm}
\begin{proof}
We introduce a radial truncation function $\chi$ such that $\chi=0$ on $B(0,1/2)$ and $\chi=1$ on ${}^cB(0,3/4)$. We have for all $k \geq 0$:
\begin{align*}
\frac{1}{N} \underset{j \neq i}{\sum} \frac{1}{d_{ij}^k} &= \int_{\mathbb{R}^3} \chi\left( \frac{x_i-y}{d_{\min}} \right) \frac{1}{|x_i-y|^k} \rho^N(t,dy)\,,\\
&=\int_{\mathbb{R}^3} \chi\left( \frac{x_i-T(y)}{d_{\min}} \right) \frac{1}{|x_i-T(y)|^k} \rho(t,dy)\,,\\
& = \left (\int_{B(x_i,3 W_\infty)} + \int_{{}^c B(x_i,3 W_\infty)}\right)\chi\left( \frac{x_i-T(y)}{d_{\min}} \right) \frac{1}{|x_i-T(y)|^k} \rho(t,dy)\,.
\end{align*}
Recall that $W_\infty \geq d_{\min}/2$. Since $\chi \left( \frac{x_i-T(y)}{d_{\min}} \right)=0$ if $|x_i-T(y)| \leq d_{\min}/2$, the first term yields:
$$
\int_{B(x_i,3 W_\infty)}\chi\left( \frac{x_i-T(y)}{d_{\min}} \right) \frac{1}{|x_i-T(y)|^k} \rho(t,dy) \leq C \|\rho\|_\infty \frac{W_\infty^3}{d_{\min}^k}.
$$
For the second term, we have $|x_i-T(y)|\geq |x_i-y|-|y-T(y)|\geq \frac{[x_i-y|}{2}$ and we get for $k \in[0,2]$:  
\begin{align*}
&\int_{{}^c B(x_i,3 W_\infty)}\chi\left( \frac{x_i-y}{d_{\min}} \right) \frac{1}{|x_i-T(y)|^k} \rho(t,dy) \\
&\leq \|\rho\|_\infty \int_{3W_\infty}^{A} \frac{1}{r^{k-2}}dr + A^{-k}\|\rho\|_{L^1},\\
&\leq \|\rho\|_\infty A^{3-k}+A^{-k},
\end{align*}
for all constant $A> 3 W_\infty$ and one can show that the optimal constant is $A=\|\rho\|_\infty^{-1/3}$ which yields the desired result. We proceed analogously for $k=3$.
\end{proof}
\section{Estimates on $\mathcal{K}^N \rho^N$, $\mathcal{K} \rho$ and control of the minimal distance}
In this part we present some estimates for the convergence of the velocity field $\mathcal{K}^N \rho^N$ and its gradient towards $\mathcal{K} \rho$ and its gradient. We estimate the $\infty$ norm of the error using the infinite Wasserstein distance between $\rho^N$ and $\rho$ in the spirit of \cite{Hauray,HJ}.\\
We recall that, according to \cite{Champion}[Theorem 5.6], at fixed time $t\geq 0$, there exists a (unique) optimal transport map $T$ satisfying :
$$
W_\infty :=W_\infty( \rho(t,\cdot),{\rho}^N (t,\cdot)) = \rho \text { - esssup } |T(x)-x|,
$$
with ${\rho}^N(t,\cdot) = T \# \rho(t,\cdot)$. This allows us to write $\mathcal{K}^N\rho^N$ as follows 
$$
\mathcal{K}^N\rho^N(x) = 6 \pi r_0 \int \chi \Phi(x-T(y)) \rho(y) dy.
$$
This important property allows us to show the following results.
\begin{prpstn}[Boundedness]\label{boundedness}
Under the assumption that $\rho \in W^{1,1}(\mathbb{R}^3) \cap W^{1,\infty}(\mathbb{R}^3)$, there exists a positive constant $C>0$ independent of $N$ such that:
$$\| \mathcal{K}^N\rho^N \|_{W^{2,\infty}} \leq C \left (1+\frac{W_\infty^3}{d_{\min}}+\frac{W_\infty^3}{d_{\min}^2}+\frac{W_\infty^3}{d_{\min}^3} \right) \|\rho\|_{W^{1,\infty}(\mathbb{R}^3) \cap W^{1,1}(\mathbb{R}^3)},$$
where 
$$
W_\infty :=W_\infty( \rho(t,\cdot), \bar{\rho}^N (t,\cdot)) = \rho \text { - esssup } |T_t(x)-x|.
$$
\end{prpstn}
\begin{rmrk}\label{hypW^3d^3}
The term $\frac{W_\infty^3}{d_{\min}^3}$ appears only for the second derivative of $\mathcal{K}^N\rho^N$ which is needed for the proof of Theorem \ref{thm_bis}.
\end{rmrk}
\begin{proof}
Let $x\in \mathbb{R}^3$, we have :
\begin{align*}
\left| \mathcal{K}^N\rho^N (x) \right| &\leq C \int \left|\chi \Phi(x-T(y)) \rho(y) dy \right|,\\
&\leq C \|\rho\|_\infty \int_{B(x,3W_\infty)}\left |\chi \Phi(x-T(y))\right| + \int_{{}^c B(x,W_\infty) } \left |\chi \Phi(x-T(y))\right||\rho(y)| dy.
\end{align*}
Recall that for all $ y \in B(x,3W_\infty)$ such that $|x-T(y)| \leq d_{\min}/2$ we have $\chi \Phi(x-T(y))=0$. Hence in all cases we have the following bound for all $y\in B(x,3W_\infty)$:
$$
\left |\chi \Phi(x-T(y))\right| \leq \frac{C}{d_{\min}},
$$
this yields the following bound 
$$
\int_{B(x,3W_\infty)} \left |\chi \Phi(x-T(y))\right| \leq C \frac{W_\infty^3}{d_{\min}}.
$$
For all y ${}^c B(x,W_\infty)$ we have that $ |x-T(y)|\geq |x-y| - |T(y) -y| \geq 2 W_\infty \geq d_{\min} $. This ensures that $\chi \Phi(x-T(y))=\Phi(x-T(y))$ on ${}^c B(x,W_\infty)$. Moreover we have 
$$ |x-T(y)| \geq |x-y| - W_\infty \geq \frac{1}{2} |x-y|,$$
which yields 
\begin{align*}
\int_{{}^c B(x,W_\infty) }\left |\chi \Phi(x-T(y))\right||\rho(y) dy& \leq C \|\rho\|_\infty \int_{{}^c B(x,W_\infty) \cap B(x,1) } \frac{dy}{|x-y|}+ \|\rho\|_{L^1},\\
&\leq C \|\rho\|_{L^1(\mathbb{R}^3) \cap L^\infty(\mathbb{R}^3)}.
\end{align*}
Analogously we obtain a similar bound for $\nabla \mathcal{K}^N$. We focus now on the bound for $\nabla^2 \mathcal{K}^N \rho^N$. We have
\begin{equation*}
\left| \nabla ^2 \mathcal{K}^N\rho^N (x)  \right|
\leq C \|\rho\|_\infty \int_{B(x,3W_\infty)}\left | \nabla ^2\chi \Phi(x-T(y))\right|dy  + \left |\int_{{}^c B(x,W_\infty) }  \nabla ^2 \chi \Phi(x-T(y))\rho(y) dy\right|.
\end{equation*}
We use the same estimates as before to bound the first term by $\|\rho\|_\infty \frac{W_\infty^3}{d_{\min}^3}$. For the second term we write
\begin{multline}\label{terme_app_1}
\left | \int_{{}^c B(x,W_\infty) }  \nabla ^2 \chi \Phi(x-T(y))\rho(y) dy \right| \leq \left |\int_{{}^c B(x,W_\infty) }  \nabla ^2 \Phi(x-y)\rho(y) dy\right|\\
+ \int_{{}^c B(x,W_\infty) } \left | \nabla ^2 \chi \Phi(x-T(y))- \nabla ^2 \Phi(x-y)\right| |\rho(y)| dy.
\end{multline}
Using an integration by parts for the first term in the right hand side of \eqref{terme_app_1} we get 
\begin{align*}
\left |\int_{{}^c B(x,W_\infty) }  \nabla ^2 \Phi(x-y)\rho(y) dy\right| & \leq \left |\int_{{}^c B(x,W_\infty) }  \nabla  \Phi(x-y) \nabla \rho(y) dy\right|\\
& + \int_{\partial B(x,W_\infty)} \left |\nabla \Phi(x-y) \right | | \rho(y)| d\sigma(y)\,,\\
& \leq C \|\nabla \rho\|_{L^1(\mathbb{R}^3) \cap L^\infty(\mathbb{R}^3) }+ \|\rho\|_\infty.
\end{align*}
Finally, for the second term in the right hand side of \eqref{terme_app_1} we have
\begin{align*}
\phantom{=}& \int_{{}^c B(x,W_\infty) } \left | \nabla ^2 \chi \Phi(x-T(y)) - \nabla^2 \Phi(x-y) \right||\rho(y)| dy  \,\\
&\leq  \int_{{}^c B(x,W_\infty) } \left ( \frac{1}{|x-y|^4}+\frac{1}{|x-T(y)|^4} \right)|y-T(y)||\rho(y)| dy \, \\
& \leq C \|\rho\|_{L^1(\mathbb{R}^3) \cap L^\infty(\mathbb{R}^3)}.
\end{align*}
\end{proof}
The following convergence estimates are used in the proof of Theorem \ref{thm_bis}.
\begin{prpstn}[Convergence estimates]\label{inifnite_cv}
The following estimates hold true: 
\begin{eqnarray*}
\| \mathcal{K}^N\rho^N - \mathcal{K} \rho \|_{L^\infty} & \lesssim & \|\rho\|_\infty W_\infty(\rho^N,\rho)\left(1+ \frac{W_\infty(\rho^N,\rho)^2}{d_{\min}} \right), \\
\|\nabla \mathcal{K}^N\rho^N - \nabla \mathcal{K} \rho \|_{L^\infty} & \lesssim &\|\rho\|_\infty W_\infty(\rho^N,\rho)\left ( | \log W_\infty(\rho^N,\rho)| + \frac{W_\infty(\rho^N,\rho)^2}{d_{\min}^2}+1 \right).
\end{eqnarray*}
\end{prpstn}
\begin{proof}
We use in the proof the shortcut $W_\infty:= W_\infty(\rho^N,\rho)$. Let $x \in \mathbb{R}^3 $, we have 
\begin{align*}
\left |\mathcal{K}^N \rho^N(x) - \mathcal{K} \rho(x) \right| & \leq 6 \pi r_0 \int_{\supp \rho} \left | \chi \Phi(x-T(y)) - \Phi(x-y) \right | \rho(y) dy.
\end{align*}
We split the integral into two disjoint domains $J:= \{ y \in \supp \rho \,,\, |x-y| \leq 3 W_\infty\}$ and its complementary. Note that on $J$, according to the definition of the truncation function $\chi$, we have  $\chi \Phi(x-T(y))=0$ for all $ y \in J$ such that $ |x-T(y)| \leq \frac{d_{\min}}{4}$ . We can then bound directly the first integral as follows
\begin{align*}
\int_{J} \left | \chi \Phi(x-T(y)) - \Phi(x-y) \right | \rho(y) dy& \leq \int_{J} \left | \chi \Phi(x-T(y))\right| \rho(y) dy+ \int_J  \left |\Phi(x-y) \right | \rho(y) dy\\
& \lesssim  \|\rho\|_\infty \left (|B(x,3W_\infty)| \frac{4}{d_{\min}} + \int_{B(x,3W_\infty)}\frac{1}{|x-y|} dy\right).
\end{align*}
Direct computations yields 
$$
\int_{J} \left | \chi \Phi(x-T(y)) - \Phi(x-y) \right | \lesssim \|\rho\|_\infty \left ( \frac{W_\infty^3}{d_{\min}} + W_\infty^2 \right ).
$$
We focus now on the remaining term, note that for all $ y\in {}^c J:={}^c B(x,3W_\infty)$ we have 
$$
|x-T(y)| \geq |x-y| - |T(y)-y| \geq 2 W_\infty \geq d_{\min},
$$
which yields that $ \chi\Phi(x-T(y)) = \Phi(x-T(y))$ on ${}^cJ$. Moreover, we have $ |x-T(y)| \geq \frac{1}{2} |x-y|$ on ${}^cJ$. We have then
\begin{align*}
\int_{{}^cJ} \left | \chi \Phi(x-T(y)) - \Phi(x-y) \right |& = \int_{{}^cJ} \left | \Phi(x-T(y)) - \Phi(x-y) \right |,\\
&\leq K \int_{{}^cJ} \left(\frac{1}{|x-T(y)|^2}+  \frac{1}{|x-y|^2}\right) |y-T(y)| \rho(y) dy,\\
&\lesssim W_\infty \|\rho\|_\infty \int_{{}^c J} \frac{1}{|x-y|^2} dy,\\
&\lesssim W_\infty \|\rho\|_\infty.
\end{align*}
In the last line we use the fact that $\frac{1}{|x-y|^2}$ is integrable on ${}^c B(x,3W_\infty)$. The proof for the second estimate is analogous to the first one. The main difference occurs for the last estimate where the $log$ term appears. This is due to the fact that we integrate $\frac{1}{|x-y|^3}$ on ${}^c B(x,3W_\infty)$.
\end{proof}
We present now an estimate for the conservation of the particle configuration. This estimate combined with Proposition \ref{boundedness} shows that the dilution regime is conserved provided that we have a control on the infinite Wasserstein distance.
\begin{prpstn}\label{conservation_distance}
For all $1\leq i \leq N $ and $j \neq i$ we have
\begin{eqnarray*}
|\dot{\xi}_i| & \lesssim & \|\nabla \mathcal{K}^N\rho^N\|_\infty \,|{\xi}_i|+ O \left(d_{\min} \right), \\
\left |\dot{x}_+^i-\dot{x}_+^j \right| &\lesssim & \|\nabla \mathcal{K}^N\rho^N\|_\infty\, |x_+^i-x_+^j|+ |\xi_i-\xi_j| +O(R), \\
\left |\dot{\xi}_i-\dot{\xi}_j \right| &\lesssim&  \|\nabla \mathcal{K}^N\rho^N\|_\infty\, \left |{\xi}_i-{\xi}_j \right| + \|\nabla^2 \mathcal{K}^N\rho^N\|_\infty\, \left |{x}_+^i-{x}_+^j \right|+O \left(d_{\min} \right).
\end{eqnarray*}
\end{prpstn}
We remark that the conservation of the infinite Wasserstein distance, which is initially of order $\frac{1}{N^{1/3}}$, ensures the control of the particle distance. Unfortunately, due to the log term appearing in Proposition \ref{inifnite_cv} we are not able to prove the conservation in time of the infinite Wasserstein distance.
\section{Existence, uniqueness and some stability properties}
In this section we present some existence, uniqueness and stability estimates. 
\begin{prpstn}\label{stab_F}
Let $p>3$. Given $F_0\in W^{2,p}(\mathbb{R}^3)$ and $u \in L^\infty(0,T;W^{3,p}(\mathbb{R}^3))$, there exists a time $T>0$ such that $F\in L^\infty(0,T;W^{2,p}(\mathbb{R}^3))$ is the unique local solution of 
\begin{equation}\label{eq:hyp_F}
\left\{
\begin{array}{rcll}
\partial_t F+ \nabla F \cdot ( \mathbb{A}^{-1}(F) \kappa g + u ) &=& \nabla u \cdot F,\,& \text{ on $[0,T] \times \mathbb{R}^3$},\\
F(0,\cdot) &=& F_0, \,& \text{ on $ \mathbb{R}^3$.}
\end{array}
\right.
\end{equation}
We have the following stability estimates
$$
\|F_1-F_2\|_{L^\infty(0,T;W^{1,p}(\mathbb{R}^3))}\leq C_1 T \|u_1-u_2\|_{L^\infty(0,T;W^{2,p}(\mathbb{R}^3))}e^{C_2 T},
$$
with $C_1$ and $C_2$ depending on $\|\mathbb{A}^{-1}\|_{2,\infty}$, $\|u_i\|_{L^\infty(0,T;W^{3,p}(\mathbb{R}^3))}$, $\|F_i\|_{L^\infty(0,T;W^{2,p}(\mathbb{R}^3))}$.
\end{prpstn}
\begin{proof}
Since $p>3$, we have $F_0 \in W^{2,p}(\mathbb{R}^3)\hookrightarrow W^{1,\infty}(\mathbb{R}^3)$ and $ u\in W^{2,\infty}(\mathbb{R}^3)$. We can apply the existence proof analogous to the existence proof of Proposition \ref{existence_FN} to get a unique solution $F\in L^\infty (0,T;W^{1,\infty}(\mathbb{R}^3))$ for a given $T>0$. It remains to show that $F\in L^\infty (0,T;W^{2,p}(\mathbb{R}^3))$ for a finite time interval. We have for $\alpha=0,1,2$
\begin{align*}
\partial_t D^\alpha F+ \nabla D^\alpha F \left( \mathbb{A}^{-1}(F) \kappa g + u \right)&=-\nabla F \cdot D^\alpha  \left( \mathbb{A}^{-1}(F) \kappa g + u \right)+(D^\alpha \nabla u) F + (\nabla u) D^\alpha F.
\end{align*}
Multiplying by $|D^\alpha F|^{p-1}$ and integrating by parts the second term using the fact that $\div(u)=0$, we get 
\begin{align*}
&\frac{1}{p} \frac{d}{dt} \int |D^\alpha F|^{p}= \frac{1}{p} \int  |D^\alpha F|^p  \div \left (\mathbb{A}^{-1}(F) \right)+ \nabla F \cdot  |D^\alpha F|^{p-1}\left( D^\alpha \left[\mathbb{A}^{-1}(F)\right] \kappa g + D^\alpha u \right)\\
&+ (D^\alpha \nabla u) F |D^\alpha F|^{p-1} + (\nabla u) D^\alpha F |D^\alpha F|^{p-1},\\
&  \lesssim \| F\|_{2,p}^{p}\left(\| \nabla \mathbb{A}^{-1}\|_\infty \|F\|_{1,\infty} + \|\nabla u\|_\infty \right)\\
&+\|D^\alpha F\|^{p-1}\left ( \|\mathbb{A}^{-1}\|_{2,\infty} +1 \right) \bigg(\|\nabla F\|_\infty \left\{ \|\nabla F\|_p + \|\nabla F\|_\infty\|\nabla F\|_p + \|\nabla^2 F\|_p+ \|D^\alpha u\|_p \right\}\\
& + \|F\|_\infty \|D^\alpha \nabla u\|_p\bigg) .
\end{align*}
Since $\|F\|_{1,\infty}\lesssim \|F\|_{2,p}$, $\|u\|_{1,\infty}\lesssim \|F\|_{2,p}$, we get up to a constant depending on $\|\mathbb{A}^{-1}\|_{2,\infty}$
\begin{align*}
\frac{d}{dt} \| D^\alpha F\|_p^{p}& \lesssim \|D^\alpha F\|_{p}^{p} \left(\|F\|_{2,p}+\|u\|_{3,p}\right)  +\|D^\alpha F\|_{p}^{p-1}\|F\|_{2,p} \left(\|F\|_{2,p}+\|u\|_{3,p} \right).
\end{align*}
Applying Young's inequality and summing over $\alpha=0,1,2$ we get 
$$
 \|F\|_{L^\infty(0,T; W^{2,p}(\mathbb{R}^3))} \lesssim \|F_0\|_{2,p}e^{ C(p,\|F\|_{2,p}, \|u\|_{3,p} ,\|\mathbb{A}^{-1}\|_{2,\infty}) T},
$$
which shows that $F\in L^\infty(0,T;W^{2,p}(\mathbb{R}^3))$ for a finite time $T>0$. Now consider two divergence free velocity fields $u_1, u_2 \in L^\infty(0,T;W^{3,p}(\mathbb{R}^3))$ and denote by $F_i$ the solution to \eqref{eq:hyp_F}. 
 We have
\begin{multline*}
\partial_t (F_1-F_2)+ (\nabla F_1-\nabla F_2) (\mathbb{A}^{-1}(F_1) \kappa g + u_1 )\\
= \nabla F_2 \left(\mathbb{A}^{-1}(F_1)- \mathbb{A}^{-1}(F_2)+u_1-u_2\right)+ ( \nabla u_1-\nabla u_2) F_1 + (F_1- F_2) \nabla u_2.
\end{multline*}
Multiplying by $|F_1-F_2|^{p-1}$ and integrating by parts the second term in the left hand side using the divergence free property of $u$, we get
\begin{align*}
\frac{d}{dt} \|F_1-F_2\|_p^p& \lesssim \|F_1-F_2\|_p^p \left( \|\nabla \mathbb{A}^{-1}\|_\infty (\|\nabla F_1\|_\infty+\|\nabla F_2\|_\infty) +\|\nabla u_2\|_\infty \right)\\
&+\| F_1-F_2\|_p^{p-1} \|u_1-u_2\|_{2,p}(\|\nabla F_1\|_\infty+\|\nabla F_2\|_\infty).
\end{align*}
For the derivative we have 
\begin{align*}
&\partial_t (\nabla F_1-\nabla F_2)+ \nabla (\nabla F_1-\nabla F_2) (\mathbb{A}^{-1}(F) \kappa g + u_1 )\\
&= - (\nabla F_1 -\nabla F_2)(\nabla \mathbb{A}^{-1}(F_1) \nabla F_1 \kappa g+\nabla u_1) + \nabla^2 F_2 \left(\mathbb{A}^{-1}(F_1)- \mathbb{A}^{-1}(F_2)+u_1-u_2\right)\\
&+\nabla F_2 \left( \left\{ \left[\nabla \mathbb{A}^{-1}(F_1)-\nabla \mathbb{A}^{-1}(F_2) \right] \nabla F_1 +\nabla \mathbb{A}^{-1}(F_2)(\nabla F_1-\nabla F_2) \right\} \kappa g + \nabla u_1 - \nabla u_2 \right)\\
&+(\nabla^2 u_1-\nabla^2 u_2) F_1 +(\nabla u_1-\nabla u_2) \nabla F_1 + \nabla u_2(\nabla F_1-\nabla F_2)  + \nabla^2 u_2( F_1-F_2).
\end{align*}
Using the same estimates as previously, we obtain 
$$
\frac{d}{dt} \|F_1 -F_2\|_{1,p}^p \leq C_1 \| F_1-F_2\|_{1,p}^p+ C_2 \| F_1-F_2\|_{1,p}^{p-1} \| u_1-u_2\|_{2,p},
$$
where $C_1, C_2$ depend on $\|\mathbb{A}^{-1}\|_{2,\infty}$, $\|u_i\|_{3,p}$, $\|F_i\|_{2,p}$. We conclude by integrating with respect to time and apply Gronwall's inequality.
\end{proof}
\begin{prpstn}\label{stab_rho}
Let $T>0$, $p>3$. We consider $\rho_0 \in  W^{1,p}(\mathbb{R}^3)$, $u\in L^\infty(0,T; W^{3,p}(\mathbb{R}^3))$ and $F\in  L^\infty(0,T; W^{2,p}(\mathbb{R}^3))$. There exists a unique solution $\rho \in L^\infty(0,T;  W^{1,p}(\mathbb{R}^3))$ to the transport equation
\begin{equation}\label{eq:transportrho}
\left \{
\begin{array}{rcl}
\partial_t \rho + \div((\mathbb{A}^{-1}(F) \kappa g+ u )\rho )&=&0,\\
\rho(0,\cdot)&=& \rho_0,
\end{array}
\right.
\end{equation}
for all $T>0$. $\rho$ satisfies 
$$
\| \rho(t,\cdot)\|_{L^\infty(0,T;W^{1,p})} \leq  \|\rho_0\|_{1,p} e^{C t},
$$
where $C$ depends on $p$, $\|\mathbb{A}^{-1}\|_{2,\infty}$, $\|F\|_{L^\infty(0,T; W^{2,p}(\mathbb{R}^3))} $, $\|u\|_{L^\infty(0,T; W^{2,p}(\mathbb{R}^3))}$. In addition, we have the following stability estimate 
$$
\|\rho_1 -\rho_2\|_{L^\infty(0,T;L^p(\mathbb{R}^3))} \leq C_1 T \left(  \| u_1-u_2\|_{L^\infty(0,T;W^{1,p}(\mathbb{R}^3))}+ \| F_1-F_2\|_{L^\infty(0,T;W^{1,p}(\mathbb{R}^3))} \right)  e^{C_2 T},
$$
with constants depending on $\|\mathbb{A}^{-1}\|_{1,\infty}$, $\| \rho_i\|_{L^\infty(0,T; W^{1,p}(\mathbb{R}^3))}$,$\|F_i\|_{L^\infty(0,T; W^{1,p}(\mathbb{R}^3))}$.
\end{prpstn}
\begin{rmrk}\label{rem:supp_comp}
If we assume in addition that $\rho_0$ is compactly supported then classical transport theory ensures that $\rho(t,\cdot)$ is compactly supported and using the characteristic flow, which is well defined since $F$, $u \in W^{1,\infty}$, one can show that 
$$
\diam(\supp (\rho(t,\cdot))) \leq \diam(\supp (\rho_0))e^{Ct},
$$
with $C=C(\|\nabla \mathbb{A}^{-1}\|_\infty, \|\nabla F\|_{L^\infty(0,t;L^\infty(\mathbb{R}^3))}, \|\nabla u\|_{L^\infty(0,t;L^\infty(\mathbb{R}^3))})$.
\end{rmrk}
\begin{proof}
Since $g=-|g|e_3$, we have the following formula
$$
\div(\mathbb{A}^{-1}(F) \kappa g ) = -\nabla \mathbb{A}^{-1}_3(F) \cdot \nabla F \kappa |g|,
$$
where $\mathbb{A}^{-1}_3$ is the third column of $\mathbb{A}^{-1}$. Note that since $p>3$, we have the following Sobolev embedding
\begin{eqnarray}\label{embedding}
\|F\|_{1,\infty} \lesssim \|F\|_{2,p}, & \|u\|_{1,\infty} \lesssim \|u\|_{2,p},& \|\rho\|_\infty \lesssim \|\rho\|_{1,p}.
\end{eqnarray}
The idea is to apply a fixed point argument. We define the operator $A$ which maps any $\rho \in L^\infty(0,T;W^{1,p})$ to the unique density $A(\rho)$ solution of 
\begin{equation}\label{eq:pointfixe}
\partial_t A(\rho)+ \nabla A(\rho) \cdot (\mathbb{A}^{-1}(F) \kappa g + u)=  \left (\nabla \mathbb{A}^{-1}_3(F) \cdot \nabla F \kappa |g| \right ) \rho.
\end{equation}
Thanks to \eqref{embedding}, $u\in W^{1,\infty}(\mathbb{R}^3)$ and $F\in W^{1,\infty}(\mathbb{R}^3)$, hence DiPerna-Lions renormalization theory ensures the existence of $\mathcal{A}(\rho)\in L^\infty(0,T;L^p(\mathbb{R}^3))$. Multiplying \eqref{eq:pointfixe} by $|A(\rho)|^{p-1}$, integrating by parts and using Young's inequality we get
\begin{align*}
\frac{1}{p} \|A(\rho)\|_p^p  &\leq \frac{1}{p} \|\rho_0\|_p^p + \frac{1}{p}\int_0^t \|A(\rho)\|_p^p \|\mathbb{A}^{-1}\|_\infty \|\nabla F\|_\infty + \int_0^t\|\mathbb{A}^{-1}\|_\infty \|\nabla F\|_\infty \|\rho\|_p \|A(\rho)\|_p^{p-1}, \\
&\leq \frac{1}{p} \|\rho_0\|_p^p+ C\int_0^t  \left(  \frac{1}{p} \|A(\rho)\|_p^p +  \frac{1}{p}\|\rho\|_p^p + \frac{p-1}{p}\|A(\rho)\|_p^p \right),\\
&\leq  \frac{1}{p} \|\rho_0\|_p^p+ C \int_0^t  \|A(\rho)\|_p^p + \frac{C}{p} t \|\rho\|_{L^\infty(L^p)}^p
\end{align*} 
with $C=C(\|\mathbb{A}^{-1}\|_\infty, \|\nabla F\|_{L^\infty(0,T;L^\infty(\mathbb{R}^3))})$. Hence, Gronwall's inequality yields
$$
\|A(\rho)\|_p \leq (\|\rho_0\|_p+ T C \|\rho\|_p) e^{Ct}.
$$
Moreover, we have
\begin{align*}
&\partial_t \nabla  A(\rho)+ \nabla  (\nabla  A(\rho))\cdot(\mathbb{A}^{-1}(F) \kappa g + u) \\
&=- \nabla A(\rho) \nabla  (\mathbb{A}^{-1}(F) \kappa g + u)+ \nabla^2 \mathbb{A}^{-1}_3(F)\kappa |g| \nabla F \nabla F \rho \\
&+\nabla \mathbb{A}^{-1}_3(F) \kappa |g| \nabla^2 F \rho+ \nabla \mathbb{A}^{-1}_3(F) \cdot \nabla F \kappa |g| \nabla \rho.
\end{align*}
Multiplying by $|\nabla  A(\rho)|^{p-1}$ and reproducing the same computations as before we get
$$
\|\nabla A(\rho)\|_p \leq (\|\nabla \rho_0\|_p+ T C_1\|\rho\|_{1,p}) e^{C_2 t},
$$
where we used \eqref{embedding}. The constants $C_1, C_2$ depend on$\|u\|_{L^\infty(0,T; W^{2,p}(\mathbb{R}^3))}$, $\|F\|_{L^\infty(0,T; W^{2,p}(\mathbb{R}^3))}$,  $\|\mathbb{A}^{-1}\|_{2,\infty}$, $p$ and $\|\rho\|_{L^\infty(0,T; W^{1,p}(\mathbb{R}^3))}$. Gathering the two estimates we obtain
\begin{equation}\label{eq:globalex}
\| A(\rho))\|_{L^\infty(0,T;W^{1,p})} \leq (\| \rho_0\|_{1,p}+ T C_1\|\rho\|_{1,p}) e^{C_2 T}.
\end{equation}
Given $\rho_1,\rho_2$, since equation \eqref{eq:pointfixe} is linear, $A(\rho_1)-A(\rho_2)$ satisfies the same equation with $\rho_0=0$. Consequently, for $T>0$ small enough, estimate \eqref{eq:globalex} shows that the mapping $A$ is a contraction and hence there exists a unique fixed point. Estimate \eqref{eq:globalex} shows also global existence. \\
Let $u_i \in L^\infty(0,T,W^{3,p}(\mathbb{R}^3))$ and $F_i\in L^\infty(0,T,W^{2,p}(\mathbb{R}^3))$ for $i=1,2$. Denote by $\rho_i$ the unique solution to equation \eqref{eq:transportrho}. We have 
\begin{align*}
&\partial_t( \rho_1-\rho_2)+\nabla (\rho_1-\rho_2) \cdot ( \mathbb{A}^{-1}(F_1) \kappa g + u_1) \\
&=- \nabla \rho_2 \cdot \left( [ \mathbb{A}^{1}(F_1) - \mathbb{A}^{-1}(F_2)]\kappa g + u_1-u_2 \right)\\
&+ (\rho_1-\rho_2) \nabla \mathbb{A}^{-1}_3(F_1) \kappa |g| \\
&+ \rho_1 \left( \left[(\nabla \mathbb{A}^{-1}_3(F_1)-\nabla \mathbb{A}^{-1}_3(F_2) )\right ] \nabla F_1 + \nabla \mathbb{A}^{-1}_3(F_2)(\nabla F_1 - \nabla F_2 ) \right) \kappa |g|.
\end{align*}
Multiplying by $|\rho_1-\rho_2|^{p-1}$ and integrating we get
$$ 
\frac{d}{dt} \|\rho_1-\rho_2\|^{p}_p \lesssim C_1 \|\rho_1-\rho_2\|_p^p+C_2 \left( \| u_1-u_2\|_\infty + \|F_1-F_2\|_\infty+ \|\nabla F_1 - \nabla F_2 \|_p \right)\|\rho_1-\rho_2\|_p^{p-1},
$$
with constants depending on $\|\mathbb{A}^{-1}\|_{1,\infty}$, $\| \rho_i\|_{1,p}$,$\|F_i\|_{1,p}$. We conclude using again the embedding $\|F_1-F_2\|_\infty \leq C \|F_1-F_2\|_{1,p}$ and analogously for $\|u_1-u_2\|_\infty$.
\end{proof}
\section*{Acknowledgement}
The author would like to thank Matthieu Hillairet for introducing the subject and sharing his experience for overcoming the difficulties during this research.



\begin{thebibliography}{10}

\bibitem{Allaire}
G.~Allaire.
\newblock Homogenization of the navier stokes equations in open sets perforated
  with tiny holes i. abstract framework, a volume distribution of holes.
\newblock {\em Arch. Rational Mech. Anal. 113}, pages 209-259 (1991).

\bibitem{Batchelor}
G.~K.~Batchelor.
\newblock{ Sedimentation in a dilute suspension of spheres.}
\newblock{\em J. Fluid Mech. 52}, 245–268 (1972).


\bibitem{Brinkman}
H.~C.~Brinkman.
\newblock {A calculation of the viscous force exerted by a flowing fluid on a dense swarm of particles.}
\newblock{\em Flow, Turbulence and Combustion}, 1-27 (1949).


\bibitem{Champion}
T.~Champion, L.~De Pascale, and P.~Juutinen.
\newblock The $\infty$-{W}asserstein distance: local solutions and existence of
  optimal transport maps.
\newblock {\em SIAMJ. Math. Anal. 40},1-20 (2008).

\bibitem{DGR}
L.~Desvillettes, F.~Golse, and V.~Ricci.
\newblock The mean field limit for solid particles in a {N}avier-{S}tokes flow.
\newblock {\em J. Stat. Phys.}, 2008.
\bibitem{DE}
M.~Doi and S.~F.~Edwards
\newblock {\em The Theory of Polymer Dynamics}.
\newblock Oxford University Press, 1986.




\bibitem{Einstein}
{\sc A.~Einstein}. {\em Eine neue bestimmung der molek\"uldimensionen}, Ann.
  Physik., 19,  (1906), pp.~289,306.
  
\bibitem{Feuillebois}
{\sc F.~Feuillebois}.
\newblock{Sedimentation in a dispersion with vertical inhomogeneities.}
\newblock{\em J. Fluid Mech. 139}, 145–171 (1984).

\bibitem{Galdi}
{\sc G.~P. Galdi}.
\newblock {\em An introduction to the mathematical theory of the
  {N}avier-{S}tokes equations}.
\newblock Springer Monographs in Mathematics.Springer, New York,, second
  edition edition, 2011.
\newblock Steady-state problems.

\bibitem{GVH}
{\sc D.~G\'erard-Varet and M.~Hillairet}, {\em Analysis of the viscosity of
  dilute suspensions beyond einstein's formula}, Arxiv:1905.08208,  (2019).
  
\bibitem{Guazzelli&Morris}
{\sc E.~Guazzelli and J.~F. Morris}.
\newblock {\em A Physical Introduction To Suspension Dynamics}.
\newblock Cambridge Texts In Applied Mathematics, 2012.

\bibitem{Haines&Mazzucato}
{\sc B.~M. Haines and A.~L. Mazzucato}, {\em A proof of einstein's effective
  viscosity for a dilute suspension of spheres.}, SIAM J. Math. Anal., 44(3),
  (2012), pp.~[2120,2145].

\bibitem{Hasimoto}
{\sc H.~Hasimoto}.
H\newblock{ On the periodic fundamental solutions of the Stokes equations and their application to viscous flow past a cubic array of spheres.}
\newblock{\em J . Fluid Mech. 5}, 317-328 (1959).



\bibitem{Hauray}
{\sc M.~Hauray}.
\newblock Wasserstein distances for vortices approximation of {E}uler-type
  equations.
\newblock {\em Math. Models Methods Appl. Sci. 19}, pages [1357,1384], 2009.

\bibitem{HJ}
{\sc M.~Hauray and P.~E. Jabin}.
\newblock Particle approximation of {V}lasov equations with singular forces :
  propagation of chaos.
\newblock {\em Ann. Sci. \'Ec. Norm. Sup\'er. (4)}, 2015.

\bibitem{HT1}
{\sc C.~Helzel and A.~E.~Tzavaras}.
\newblock { A comparison of macroscopic models describing the collective response of sedimenting rod-like particles in shear flows}
\newblock {\em Physica D 337}, 18-29 (2016).

\bibitem{HT2}
{\sc C.~Helzel and A.~E.~Tzavaras}.
\newblock A kinetic model for the sedimentation of rod–like particles
\newblock {\em Multiscale Model. Simul. 15} (2017), 500-536.
\bibitem{Hillairet}
{\sc M.~Hillairet}.
\newblock On the homogenization of the stokes problem in a perforated domain.
\newblock {\em Arch Rational Mech Anal, 230}, pages 1179,1228, 2018.

\bibitem{HMS}
{\sc M.~Hillairet and A.~Moussa and F.~Sueur}.
\newblock On the effect of polydispersity and rotation on the {B}rinkman force
  induced by a cloud of particles on a viscous incompressible flow.
\newblock {\em arXiv:1705.08628v1 [math.AP]}, 2017.

\bibitem{HW}
{\sc M.~Hillairet and D.~Wu}, {\em Effective viscosity of a polydispersed
  suspension}, Arxiv:1905.12306,  (2019).
  
\bibitem{theseHofer}
{\sc R.~M. H\"ofer}.
\newblock {\em Sedimentation of particle suspensions in Stokes flow}.
\newblock {Ph.D thesis. Rheinischen Friedrich-Wilhelms university. Bonn}. 2019.

\bibitem{Hofer}
{\sc R.~M. H{\"o}fer}.
\newblock Sedimentation of inertialess particles in {S}tokes flows.
\newblock {\em Commun. Math. Phys. 360}, 55-101 (2018).




\bibitem{HV}
{\sc R.~M. H{\"o}fer and J.~J.~L Vel{\`a}zquez}.
\newblock The method of reflections, homogenization and screening for {P}oisson
  and {S}tokes equations in perforated domains.
\newblock {\em Arch Rational Mech Anal, 227}, pages 1165,1221, 2018.

\bibitem{JO}
{\sc P.~E Jabin and F.~Otto}.
\newblock Identification of the dilute regime in particle sedimentation.
\newblock {\em Communications in Mathematical Physics}, 2004.

\bibitem{Jeffrey&Onishi}
{\sc D.~J. Jeffrey and Y.~Onishi}.
\newblock Calculation of the resistance and mobility functions for two unequal
  spheres in low-{R}eynolds-number flow.
\newblock {\em J . Fluid Meoh. vol. 139}, 261-290 (1984).

\bibitem{KRM}
{\sc J.~B. Keller and L.~A. Rubenfeld and J.~E. Molyneux}, {\em Extremum
  principles for slow viscous flows with applications to suspensions.}, Journal
  of Fluid Mechanics, 30(1),  (1967), pp.~[97,125].
  
\bibitem{KK}
{\sc S.~Kim and S.~J. Karrila}.
\newblock {\em Microhydrodynamics : Principles and Selected Applications}.
\newblock Courier {C}orporation, 2005.

\bibitem{LLS}
{\sc P.~Laurent and G.~Legendre and J.~Salomon}.
\newblock On the method of reflections.
\newblock {\em https://hal.archives-ouvertes.fr/hal-01439871}, 2017.

\bibitem{LL}
{\sc C.~Le Bris and T.~Leli\`evre}.
\newblock Micro-macro models for viscoelastic fluids: modelling, mathematics and numerics
\newblock {\em Science China Mathematics volume 55, 353–384}, (2012).

\bibitem{LSP}
{\sc T.~L\'evy and E.~S\'anchez-Palencia}, {\em Einstein-like approximation for
  homogenization with small concentration. ii. navier-stokes equation.},
  Nonlinear Anal., 9(11),  (1985), pp.~[1255--1268].

\bibitem{Luke}
{\sc J.~H.~C. Luke}.
\newblock Convergence of a multiple reflection method for calculating {S}tokes
  flow in a suspension.
\newblock {\em Society for Industrial and Applied Mathematics}, 1989.


\bibitem{Mecherbet}
{\sc A.~Mecherbet}.
\newblock Sedimentation of particles in stokes flow.
\newblock {\em Kinetic \& Related Models 12(5),} 995-1044 (2019).

\bibitem{Niethammer&Schubert}
{\sc B.~Niethammer and R.~Schubert}, {\em A local version of einstein's formula
  for the effective viscosity of suspensions}, arXiv:1903.08554,  (2019).



\bibitem{Rubinstein}
{\sc J.~Rubinstein}.
\newblock{On the macroscopic description of slow viscous flow past a random array of spheres.}
\newblock{J. Stat. Phys. 44}, 849–863 (1986).

\bibitem{Rubinstein&Keller}
{\sc J.~Rubinstein and J.~Keller}.
\newblock{ Particle distribution functions in suspensions.}
\newblock{ \em Phys. Fluids A 1}, 1632–1641 (1989).




\bibitem{Smo}
{\sc M.~Smoluchowski}.
\newblock \"{U}ber die {W}echselwirkung von {K}ugeln, die sich in einer z\"ahen
  {F}l\"ussigkeit bewegen.
\newblock {\em Bull. Acad. Sci. Cracovie A 1}, 28-39 (1911).
\bibitem{SP}
{\sc E.~S\'anchez-Palencia}, {\em Einstein-like approximation for
  homogenization with small concentration.}, I. Elliptic problems. Nonlinear
  Anal., 9(11),  (1985), pp.~[1243--1254].
\end{thebibliography}
\end{document}